\documentclass{amsart}
\usepackage{graphicx} 

\newcommand{\gopr}{\gnecessary}

\DeclareSymbolFont{extraup}{U}{zavm}{m}{n}
\DeclareMathSymbol{\vardiamond}{\mathalpha}{extraup}{87}
 \DeclareSymbolFont{symbolsC}{U}{txsyc}{m}{n}
\DeclareMathSymbol{\strictif}{\mathrel}{symbolsC}{74}
\DeclareMathSymbol{\strictfi}{\mathrel}{symbolsC}{75}
\DeclareMathSymbol{\strictiff}{\mathrel}{symbolsC}{76}
\newcommand{\tto}{\strictif}

\renewcommand{\descriptionlabel}[1]%
{\hspace{\labelsep}\emph{#1:}}

\usepackage{amssymb}
\usepackage{latexsym}
\usepackage{amsmath}
\usepackage{enumerate}
\usepackage{amsthm}
\usepackage{wasysym}
\usepackage{stmaryrd}
\usepackage{mathrsfs}
\usepackage{pifont}
\usepackage{tikz}
\usepackage[f]{esvect}
\usepackage[normalem]{ulem}

\newcommand{\qee} {\hspace*{2mm}\hfill \ding{109}}

\renewcommand{\iff}{\leftrightarrow}
\renewcommand{\leq}{\leqslant}
\renewcommand{\geq}{\geqslant}
\renewcommand{\preceq}{\preccurlyeq}

\renewcommand{\phi}{\varphi}

\renewcommand{\Theta}{\varTheta}
\renewcommand{\Phi}{\varPhi}
\renewcommand{\Psi}{\varPsi}
\renewcommand{\Xi}{\varXi}
\renewcommand{\Omega}{\varOmega}
\renewcommand{\Gamma}{\varGamma}

\newtheorem{theorem}{Theorem}[section]
\newtheorem{define}[theorem]{Definition}

\newtheorem{exa}[theorem]{Example}

\newtheorem{exerc}[theorem]{Exercise}

\newtheorem{conj}[theorem]{Conjecture}

\newtheorem{ques}[theorem]{Open Question}
\newenvironment{question}{\begin{ques} \rm}{\qee\end{ques}}
\newtheorem{lem}[theorem]{Lemma}

\newtheorem{cor}[theorem]{Corollary}

\newtheorem{rem}[theorem]{Remark}
\newenvironment{remark}{\begin{rem} \rm}{\qee\end{rem}}
\newtheorem{cl}[theorem]{Claim}
\newenvironment{claim}{\begin{cl} \rm}{\qee\end{cl}}

\definecolor{uuxgreen}{cmyk}{1,0,0.75,0}
\definecolor{uuxred}{cmyk}{0.2,1,0.9,0.1}
\definecolor{uuyblue}  {cmyk}{0.9,0.55,0,0}
\definecolor{uuxblue}  {cmyk}{0.9,0.55,0,0}

\DeclareMathOperator{\possible}{\text{\tikz[scale=.6ex/1cm,baseline=-.6ex,rotate=45,line width=.1ex]{
                            \draw (-1,-1) rectangle (1,1);}}}
\DeclareMathOperator{\necessary}{\text{\tikz[scale=.6ex/1cm,baseline=-.6ex,line width=.1ex]{
                            \draw (-1,-1) rectangle (1,1);}}}

 \DeclareMathOperator{\gnecessary}{\text{\tikz[scale=.6ex/1cm,baseline=-.6ex,line width=.1ex]{
                            \draw[gray, fill = gray, fill opacity = .90] (-1,-1) rectangle (1,1);}}}

\newcommand{\gnum}[1]{{\ulcorner #1 \urcorner}}
\newcommand{\mc}[1]{\mathcal #1}
\newcommand{\num}[1]{{\underline {#1}}}
\newcommand{\mf}[1]{{\mathfrak {#1}}}
\newcommand{\verz}[1]{\{ #1 \}}
\newcommand{\tupel}[1]{{\langle #1 \rangle}}
\newcommand{\apr}{{\vartriangle}}
\newcommand{\aco}{{\triangledown}}
\newcommand{\opr}{\necessary}
\newcommand{\oco}{\possible}

\newcommand{\bleq}{\mathbin{\leq}}
\newcommand{\bgeq}{\mathbin{\geq}}
\newcommand{\qedright}{\belowdisplayskip=-12pt}

\newcommand{\PA}{{\sf PA}}
\newcommand{\EA}{\mathsf{EA}}
\newcommand{\PR}{\mathrm{Pr}}

\newcommand{\Bews}{\mathrm{Proof}^\ast}
\newcommand{\Prov}{\mathrm{Prov}}
\newcommand{\Proof}{\mathrm{Proof}}
\newcommand{\Con}{\mathrm{Con}}
\newcommand{\Ass}{\mathrm{Ass}}
\newcommand{\True}{\mathrm{True}}
\newcommand{\Sent}{\mathrm{Sent}}
\newcommand{\Fml}{\mathrm{Fml}}
\newcommand{\Bell}{\mathrm{Bell}}

\newcommand{\strind}{Strong Independence}
\newcommand{\strynd}{strongly independent}
\newcommand{\eac}{\ensuremath{{\sf EA}+{\mathrm B}\Sigma_1}}
\newcommand{\grullet}{\textcolor{gray}{$\bullet$}}

\title{Extensional Independence}

\author{Taishi Kurahashi}
\address{Graduate School of System Informatics, Kobe University, Japan.}
\email{kurahashi@people.kobe-u.ac.jp}
\author{Albert Visser}
 \address{Philosophy, Faculty of Humanities,
                Utrecht University,
               Janskerkhof 13,
                3512BL~~Utrecht, The Netherlands}
\email{a.visser@uu.nl}
\date{\today}



\thanks{The first author is supported by JSPS KAKENHI Grant Number JP23K03200.}

\begin{document}

\begin{abstract}
Joel Hamkins asks whether there is a $\Pi^0_1$-formula $\rho(x)$ such that
$\rho(\gnum \phi)$ is independent over ${\sf PA}+\phi$, if this theory is consistent, where
this construction is extensional in $\phi$ with respect to {\sf PA}-provable equivalence. 
We show that there can be no such extensional Rosser formula of any complexity.

We give a positive answer to Hamkins' question for the case where we replace
Extensionality by a weaker demand that we call \emph{Consistent
Extensionality}. We also prove that we can demand the negation of
$\rho$ to be $\Pi^0_1$-conservative, if we ask for the still weaker
\emph{Conditional Extensionality}.

We show that an intensional version of the result for 
Conditional Extensionality
cannot work.
\end{abstract}

\maketitle

\section{Introduction}

Rosser's version of the First Incompleteness Theorem tells us far more than the bare fact that Peano Arithmetic $\PA$ is incomplete.
It tells us that \PA\ is \emph{effectively essentially incomplete} with $\Pi_1$ witnesses. This means that from an index of a consistent 
computably enumerable (c.e.) extension $U$ of \PA\ we can effectively find a $\Pi_1$-sentence $\rho$ that is independent of $U$, that is, both 
$U + \rho$ and $U + \neg\, \rho$ are consistent. 
We see that we have the following components of the result.
\begin{description}
\item[Essentiality] Not just the base theory is incomplete but all extensions from a suitable given class.
\item[Effectivity] We can find the witnesses of independence effectively from an index of the extending theory.
\item[Restriction] The witnesses are in a certain prescribed class. 
\end{description}

The Rosser Theorem seems better than the original First Incompleteness Theorem:
for G\"odel's theorem to work the extensions need to be $\Sigma_1$-sound. However, 
in one important respect, the Rosser version falls short of the
original result: \emph{extensionality}. If we restrict ourselves to $\Sigma_1$-sound extensions of \PA\ with a sentence
$\phi$, we can view the mapping $\Phi$ from $\phi$ to an independent sentence of $\PA+\phi$ as an operation on the
Lindenbaum algebra of \PA. This means that $\PA \vdash \phi \iff \psi$ implies $\PA\vdash\Phi(\phi) \iff \Phi(\psi)$. 
In other words, the operator is extensional.
Since the algebraic perspective constitutes a fundamental way of looking at a theory, this extra property is of intrinsic interest.

We consider these notions in a bit more detail. We restrict our attention to finite extensions.
We start with essentiality. 
We say that a theory $T$ is \emph{f-essentially incomplete} if, for every sentence $\varphi$, whenever $T+\varphi$ is consistent, it is incomplete (cf.~\cite{viss:pour24}). 
Here f stands for `finite'.
This notion admits a natural algebraic characterisation: $T$ is f-essentially incomplete if and only if its Lindenbaum algebra has no atoms. 
We note that this result holds also for theories that are not computably enumerable, however our default will be to consider c.e.~theories.
Since all countable atomless Boolean algebras are isomorphic (cf.~\cite[Chapter 16]{giva:intro09}), 
the Lindenbaum algebras of any two consistent c.e.~theories that are f-essentially incomplete are isomorphic. 
Also, such Boolean algebras are dense. 

As in the case of essential incompleteness, f-essential incompleteness has effective variants. 
There are two ways to formulate effectiveness. 
The first formulation is intensional.
Let $\mathsf{W}_i$ denote the c.e.\ set with index $i$.
We say that $T$ is \emph{effectively if-essentially incomplete} if there exists
a computable partial function $f$ such that, for any index $i$, if $\mathsf{W}_i$
is a consistent finite extension of $T$, then the value $f(i)$ is defined and is
a sentence that is independent of $\mathsf{W}_i$.
The second formulation is extensional.
We say that $T$ is \emph{effectively ef-essentially incomplete} if there exists
a computable partial function $f$ such that, for any sentence $\varphi$, if
$T+\varphi$ is consistent, then the value $f(\varphi)$ is defined and is
independent of $T+\varphi$.
For a detailed analysis of these effective variants, see~\cite{viss:pour24}.

These two effective notions are closely related.
Pour-El \cite{pour:effe68} showed that effective if-essential incompleteness coincides with
effective essential incompleteness. In the same paper, Pour-El also showed that effective essential incompleteness
coincides with effective inseparability. By a result of Pour-El and Kripke \cite{pour:dedu67}, all effectively inseparable
theories are recursively boolean isomorphic. It follows that these theories are precisely the equivalence class modulo recursive boolean isomorphism of, say, {\sf R}.
This equivalence class contains all foundational theories like ${\sf PA}^-$, ${\sf S}^1_2$, {\sf EA}, $\mathrm I\Sigma_1$, {\sf PA},
${\sf ACA}_0$, {\sf ZF}, {\sf GB}, \dots

Effective ef-essential incompleteness does not, in general, imply
effective if-es\-sen\-tial incompleteness.
Indeed, there exist decidable theories that are effectively ef-es\-sen\-tially
incomplete but not essentially incomplete \cite[p.~600]{jone:effe70}, \cite{murw:noeu24}, \cite[Example 4.3]{viss:pour24}.
However, the situation changes when the ef-version is formulated in terms of 
formulas.
In the context of effective ef-essential incompleteness, if the independent sentences obtained by $f$ are always 
$\Pi_n$-sentences, then by using a $\Sigma_1$-formula $\delta(x,y)$ representing $f$, 
one can define a formula \[\rho(x) : \equiv 
\forall y\, \bigl(\delta(x,y)\rightarrow \mathrm{True}_{\Pi_n}(y)\bigr),\] so that $\PA\vdash f(\varphi)\leftrightarrow \rho(\ulcorner\varphi\urcorner)$.
In this case, the fixed point theorem becomes available, and one can find a computable function witnessing effective if-essential incompleteness based on $\rho(x)$. 
That is, there exists a computable function $f$ such that, if $\mathsf{W}_i$ is a consistent finite extension of $\PA$, 
then there is a sentence $\varphi$ with $f(i) = \rho(\ulcorner\varphi\urcorner)$, and $f(i)$ is independent of $\mathsf{W}_i$ \cite{viss:pour24}.

We now turn to the question of extensionality, outlining what we will do in this paper.
For the effective versions of f-essential incompleteness introduced above, it
is natural to ask whether the operation of taking an independent sentence can
be required to respect provable equivalence.
More precisely, we ask whether independent sentences can be obtained in a way
that depends only on the Lindenbaum algebra of the underlying theory, rather
than on a particular presentation of the theory.
This question is closely connected with the question of the density of the Lindenbaum algebra of theories like $\PA$. 
Consider the Lindenbaum algebra of \PA.
Suppose that $[\varphi] < [\psi]$, that is, $\PA\vdash \varphi\rightarrow\psi$
but $\PA\nvdash \psi\rightarrow\varphi$.
Then one can effectively find a sentence $\xi_{\varphi, \psi}$ such that $[\varphi] < [\xi_{\varphi, \psi}] < [\psi]$ from $\varphi$ and $\psi$ 
by considering a sentence that is independent of the consistent theory $\PA+\psi+\neg\, \varphi$. 
However, the choice $\xi_{\varphi, \psi}$ depends on $\varphi$ and $\psi$, and hence even if $[\varphi]=[\varphi']$ and $[\psi]=[\psi']$, 
the resulting sentences $\xi_{\varphi,\psi}$ and $\xi_{\varphi',\psi'}$ need not be $\PA$-provably equivalent.
This leads to the following question.
Can independent sentences be obtained effectively in a way that reflects only
the extensional structure of the Lindenbaum algebra?
Shavrukov and Visser gave an affirmative answer to this question by proving the following result on \emph{Extensional Independence}. 

\begin{theorem}[Shavrukov and Visser \cite{shav:unif14}]\label{thm:SV_independence}
There is a $\Delta^0_2$-formula $\rho(x)$ over {\sf PA} with the following properties.
\begin{description}
\item[Independence] If ${\sf PA}+ \phi$ is consistent, then so are ${\sf PA}+ \phi+\rho(\gnum\phi)$ and\\
 ${\sf PA}+ \phi+\neg\, \rho(\gnum\phi)$.
\item[Conditional Extensionality] If ${\sf PA} \vdash \phi \iff \psi$, then $ {\sf PA}+ \phi \vdash\rho(\gnum \phi) \iff \rho(\gnum\psi)$.
\end{description}
\end{theorem}

By letting $f_\rho(\varphi) : = \rho(\ulcorner\varphi\urcorner)\land\varphi$,
we obtain an extensional version of effective ef-essential incompleteness.

\begin{cor}
There is a computable function $f_\rho$ with the following properties.
\begin{description}
\item[Independence] If ${\sf PA}+ \phi$ is consistent, then so are the theories ${\sf PA}+ \phi+ f_\rho(\phi)$ and ${\sf PA}+ \phi+\neg\, f_{\rho}(\phi)$.
\item[Extensionality] If ${\sf PA} \vdash \phi \iff \psi$, then $ {\sf PA} \vdash f_{\rho}(\phi) \iff f_{\rho}(\psi)$.
\end{description}
\end{cor}
\noindent In other words, Conditional Extensionality for $\rho$ entails Extensionality for $f_\rho$.
Using this construction, the density of the Lindenbaum algebra of $\PA$ can be
witnessed in an extensional manner.
Let $[\varphi] \leq [\psi]$ denote $\PA \vdash \varphi \to \psi$. 
Also, $[\varphi] < [\psi]$ and $[\varphi] = [\psi]$ denote ($[\varphi] \leq [\psi]$ and $[\psi] \not \leq [\varphi]$) and ($[\varphi] \leq [\psi]$ and $[\psi] \leq [\varphi]$), respectively. 

\begin{theorem}[Shavrukov and Visser \cite{shav:unif14}]\label{thm:SV_density}
There is a computable function $g_\rho$ with the following properties. 
\begin{description}
\item[Density] If $[\varphi] < [\psi]$, then $[\varphi] < [g_\rho(\varphi,\psi)] < [\psi]$.
\item[Extensionality] If $[\varphi]=[\varphi']$ and $[\psi]=[\psi']$, then $[g_\rho(\varphi,\psi)] = [g_\rho(\varphi',\psi')]$.
\end{description}
\end{theorem}

Inspired by the result of Shavrukov and Visser, in his paper \cite{hamk:nonl22}, Joel Hamkins asks the following natural question (Question 28).\footnote{For notational consistency with the rest of this paper we changed the variable-names in the formulation of the question.}
Is there a $\Pi^0_1$-formula $\rho(x)$ with the following properties?
\begin{description}
\item[Independence]  If ${\sf PA}+ \phi$ is consistent, then so are ${\sf PA}+ \phi+\rho(\gnum\phi)$ and\\
 ${\sf PA}+ \phi+\neg\, \rho(\gnum\phi)$.
\item[Extensionality] If ${\sf PA} \vdash \phi \iff \psi$, then  $ {\sf PA} \vdash\rho(\gnum \phi) \iff \rho(\gnum\psi)$.
\end{description}

\noindent
We call a formula that satisfies both Independence and Extensionality \emph{an extensional Rosser formula}.
We allow an extensional Rosser formula to be of any complexity.

Relative to Shavrukov {\&} Visser's result, we see that Hamkins is asking for two improvements:
(i) bring down $\Delta^0_2$ to $\Pi^0_1$ and (ii) improve Conditional Extensionality to
Extensionality. 
Hamkins' question can be understood as asking the following.
Can one effectively assign a $\Pi_1$-sentence independent of $U$ to $U$ in a way that does not depend on the presentation of $U$?
Equivalently, is the effective operation of taking an independent sentence for $U$ an extensional property of $U$, or is it necessarily intensional?

We answer Hamkins' question negatively in Section~\ref{nerf}. 
There simply is no extensional Rosser formula (of any complexity), not just over {\sf PA},
but over a wide range of theories (Theorem \ref{grotesmurf}).  
In a sense, this result shows that Hamkins' question was
not quite the right question.

In Section~\ref{csorf}, we will show that, if we weaken Extensionality to \emph{Consistent Extensionality}, Hamkins' question gets a positive
answer (Theorem \ref{con_ext}). 
Specifically, we construct a $\Pi^0_1$-formula $\rho(x)$ such that we have the following.
\begin{description}
\item[Independence]  If ${\sf PA}+ \phi$ is consistent, then so are ${\sf PA}+ \phi+\rho(\gnum\phi)$ and\\
 ${\sf PA}+ \phi+\neg\, \rho(\gnum\phi)$.
\item[Consistent Extensionality] If ${\sf PA}+ \phi$ is consistent and 
${\sf PA} \vdash \phi \iff \psi$, then  $ {\sf PA} \vdash\rho(\gnum \phi) \iff \rho(\gnum\psi)$.
\end{description}
This shows that, although extensionality in the sense asked by Hamkins is impossible, effective ef-essential incompleteness with $\Pi^0_1$-formulas can be witnessed by an extensional way. 
In fact our solution, does not just work for {\sf PA}, but for all c.e.~extensions of Elementary
Arithmetic, {\sf EA}, also known as $\mathrm I\Delta_0+{\sf Exp}$ (cf.~\cite{haje:meta91})

In Section~\ref{cdorf}, we discuss a third form of extensionality.
If we weaken \emph{Consistent Extensionality} to \emph{Conditional Extensionality} as in Shavrukov and Visser's Theorem \ref{thm:SV_independence}, we can strengthen  
\emph{Independence} to, say, \emph{\strind} (Theorem \ref{maintheorem}). 
Specifically, we can show the following. There is a $\Pi^0_1$-formula $\rho(x)$ satisfying the following properties.
\begin{description}
\item[\strind] If ${\sf PA}+ \phi$ is consistent, then ${\sf PA}+\phi +\rho(\gnum\phi)$ is 
consistent and 
${\sf PA}+\phi+\neg\,\rho(\gnum \phi)$ is
$\Pi^0_1$-conservative over ${\sf PA}+\phi$.
\item[Conditional Extensionality] If ${\sf PA} \vdash \phi \iff \psi$, then 
$ {\sf PA}+ \phi \vdash\rho(\gnum \phi) \iff \rho(\gnum\psi)$.
\end{description}
As a consequence of this theorem, we obtain a computable function $f(\cdot, \cdot)$ satisfying the following conditions, 
which substantially strengthen Theorem~\ref{thm:SV_density} on extensional density of the Lindenbaum algebra (Theorem \ref{Lindenbaum}): 
\begin{description}
	\item[Strong Density] If $[\varphi] < [\psi]$ and $\varphi, \psi \in \Gamma \supseteq \Pi_1$, then we 
	have $f(\varphi, \psi) \in \Gamma$, $[\varphi] < [f(\varphi, \psi)] < [\psi]$, and for any $\Sigma_1$ sentence $\sigma$, if $[\sigma] \leq [f(\varphi, \psi)]$, then $[\sigma] \leq [\varphi]$. 
	\item[Extensionality] If $[\varphi] = [\varphi']$ and $[\psi] = [\psi']$, then $[f(\varphi, \psi)] = [f(\varphi', \psi')]$.  
\end{description}

Finally, in Section~\ref{inerf}, we prove a further negative result.
We show that, in the intensional case, we have a much stronger negative result, to wit, 
    there is \emph{no} computable function $\Phi$ satisfying the following condition (Corollary \ref{if_cor2}). 
    For any natural numbers $i$ and $j$, if $\mathsf{W}_i$ and $\mathsf{W}_j$ are deductively equivalent 
        consistent finite extensions of $U$, then $\Phi(i)$ and $\Phi(j)$ converge to sentences, 
        \begin{description}
        \item[Weak Independence]
        $U \nvdash \Phi(i)$ and $\mathsf{W}_i \nvdash \neg\, \Phi(i)$;
         \item[Conditional Extensionality]        
         $\mathsf{W}_i \vdash \Phi(i) \leftrightarrow \Phi(j)$. 
\end{description}
Thus, in this sense, effective if-essential incompleteness is inherently intensional, that is, any effective choice of an independent sentence of a finite extension must essentially depend on the particular c.e.~presentation of the finite extension.

\subsection{History of the Paper}
The present paper supersedes Albert Visser’s preprint \emph{On a Question of Hamkins'}, incorporating and extending its content.

\section{Provability predicates and related background}
In this section, we briefly review provability predicates, Rosser provability predicates, witness comparison, partial conservativity, and Feferman-style provability predicates. 
We also evaluate Shavrukov and Visser's formula in Theorem \ref{thm:SV_independence}.

For each consistent c.e.~extension $U$ of $\EA$, we find an elementary formula that is a proof predicate $\Proof_U(x, y)$ saying that $y$ is a $U$-proof of $x$, that is, $U \vdash \varphi$ if and only if $\Proof_U(\gnum{\varphi}, \num p)$ for some natural number $p$. 
The existence of such a formula is guaranteed by Craig's trick. 
Let $\Prov_U(x)$ be the $\Sigma^0_1$-formula $\exists y\, \Proof_U(x, y)$ which is a provability predicate of $U$. 
Given such a proof predicate $\Proof_U(x, y)$, one can define a Rosser provability predicate $\Prov_U^{\mathrm{R}}(x)$ by 
\[
    \exists y\, (\Proof_U(x, y) \land \forall z \leq y\, \neg\, \Proof_U(\neg \, x, z)). 
\]
A $\Pi^0_1$ fixed point of the negation of such a predicate is called a $\Pi^0_1$-Rosser sentence of $U$, which is independent of $U$.
Different choices of proof predicates provide different Rosser provability predicates.

  It has been shown in previous studies that the properties of Rosser sentences for a consistent theory $U$ depend on the way how such a Rosser provability predicate is constructed. 
  Guaspari and Solovay \cite{guas:ross79} pioneered such work, and in particular they proved that the uniqueness of Rosser sentences of $U$ modulo $U$-provability depends on the construction of Rosser provability predicates. 
  Various questions can be asked about $U$-independent sentences and, thus,
  the work of Guaspari and Solovay has been followed by the development of various methods for constructing 
  provability predicates with properties of various kinds. 
  In Section \ref{csorf}, we prove Theorem \ref{con_ext} by constructing a Rosser provability predicate, using a
  method developed by Guaspari and Solovay, so that $\rho(\gnum{\phi})$ is a Rosser sentence of $\PA + \varphi$.   

    Guaspari and Solovay also developed the systematic analysis of witness comparison, which plays a central role in the construction of Rosser provability predicates. 
    For formulas $\varphi \equiv \exists x \,\alpha (x)$ and $\psi \equiv \exists x \,\beta(x)$, we define the formulas $\varphi < \psi$ and $\varphi \leq \psi$ as follows: 
\begin{itemize}
\item 
$\varphi < \psi : \equiv \exists x \, (\alpha(x) \wedge \forall y \leq x \, \neg \,\beta(y))$, 
\item 
$\varphi \leq \psi : \equiv \exists x \, (\alpha(x) \wedge \forall y < x \, \neg\, \beta(y))$.
\end{itemize}
A $\Pi^0_1$-Rosser sentence $\rho$ of $U$ is given by $\EA \vdash \rho \leftrightarrow \neg\, (\Prov_U(\gnum{\rho}) < \Prov_U(\gnum{\neg\, \rho}))$. 
We will freely use the following property of witness comparison formulas without referring to it: 

\begin{lem}[cf.~Lindstr\"om {\cite[Lemma 1.3]{lind:aspe03}}]\label{lem:witness}
For any formulas $\varphi \equiv \exists x \, \alpha(x)$ and $\psi \equiv \exists x \, \beta(x)$, where $\alpha(x)$ and $\beta(x)$ are elementary formulas, we have the following clauses: 
\begin{enumerate}[1.]
\item 
$\EA \vdash \varphi < \psi \to \varphi \leq \psi$.
\item 
$\EA \vdash \varphi \vee \psi \to (\varphi < \psi) \vee (\psi \leq \varphi)$.
\item 
$\EA \vdash \neg \, \bigl((\varphi < \psi) \wedge (\psi \leq \varphi) \bigr)$.
\item 
$\EA \vdash \varphi \wedge \neg \, \psi \to \varphi < \psi$.
\end{enumerate}
\end{lem}
 
  The systematic study of sentences with partial conservativity was initiated by Guaspari \cite{guas:part79}. 
  For a class $\Gamma$ of formulas, we say that a sentence $\varphi$ is $\Gamma$-conservative over a theory $U$ if $U$ proves a $\Gamma$ sentence $\gamma$ whenever $U + \varphi$ proves $\gamma$.   
  Guaspari proved that for each consistent c.e.~extension $U$ of $\PA$ and for each $n \geq 1$, there exists a $\Pi^0_n$ (resp.~$\Sigma^0_n$)-sentence $\gamma$ such that $U + \gamma$ is consistent and $U + \neg \, \gamma$ is $\Pi^0_n$ (resp.~$\Sigma^0_n$)-conservative over $U$. 
  For each $\Pi^0_1$-Rosser sentence $\xi$ for $U$, it is known that $\neg\, \xi$ is not $\Pi^0_1$-conservative over $U$ (see \cite[Exercise 5.1]{lind:aspe03}), and so Guaspari's sentences cannot be Rosser sentences. 
  So, our method proving Theorem \ref{maintheorem} in Section \ref{cdorf} is substantially different from that of Section \ref{csorf}. 
 
 The formula employed by Shavrukov {\&} Visser in Theorem \ref{thm:SV_independence} is a meaningful formula. 
 It generates Smory\'nski's  Rosser sentences from \cite{smor:self89}.
 Moreover, it generates  fixed-point-free sentences that satisfy (uniquely modulo provable equivalence) the 
 G\"odel equation for a certain Feferman-stye provability predicate. 
 See  \cite{smor:self89} and \cite{shav:smar94}. 
 
 We remind the reader that a \emph{Fefermanian provability predicate} is a provability predicate that is constructed from
 a fixed formalisation of \emph{provability-from-assumptions} in combination with a representation of the axiom set. The idea
 is that the only \emph{funny business} can happen in the representation of the axiom set.

 We also have the notion of a \emph{Feferman-style provability predicate}
 for a theory $U$. Such a predicate  is a Fefermanian provability predicate, where the axiom set is built from an elementary representation of the axiom set plus the stipulation that
 we may only use an axioms if the set of axioms smaller or equal to its G\"odel number is consistent. See \cite{fefe:arit60}.\footnote{Of course, there are many variations possible on 
 our definition of the Feferman-style predicate. For example, we can demand that we may use an axiom if it and the axioms enumerated before it are consistent.}
 Thus, the generated sentences, while not being consistency statements, do share two important properties with consistency
 statements: being explicit, i.e., fixed-point-free, and being unique solutions of a G\"odel equation. 
 Since $\PA$ is essentially reflexive, that is, every consistent extension of $\PA$ proves the consistency of every finite subtheory of it, 
 a Feferman-style provability predicate is actually a provability predicate. 

 We have the following insight. 
\begin{theorem}
Let $U$ be any c.e.~extension of {\sf PA}. Consider Feferman-style provability for any enumeration of $U$.
Let $\gamma$ be a G\"odel sentence for Feferman provability. Then, $\gamma$ and $\neg\, \gamma$ are both
$\Pi^0_1$-conservative over $U$.
\end{theorem}

\noindent
We give a direct proof. Since, in extensions of {\sf PA}, $\Pi^0_1$-conservativity and interpretability coincide, 
we can also give a proof in an appropriate interpretability logic. See \cite[p.~177]{viss:pean89}.
 
  \begin{proof}
      We write $\vdash$ for $U$-provability and $\apr$ for our Feferman-style provability predicate.
      Then, $\vdash \gamma \leftrightarrow \neg \, \apr \gamma$ and $\vdash \neg \, \apr \bot$. Let $\pi$ be any
      $\Pi^0_1$-sentence and let $\sigma$ be a $\Sigma^0_1$-sentence {\sf PA}-provably equivalent to $\neg\,\pi$.  

      Suppose $\vdash \gamma \to \pi$. Then, $\vdash \sigma \to \apr\gamma$. It follows that
      $\vdash \sigma \to \apr(\gamma \wedge \sigma)$ and, hence, $\vdash \sigma\to \apr(\gamma \wedge \neg\, \gamma)$.
      So, $\vdash \sigma\to \apr\bot$. We may conclude $\vdash \pi$.

      Suppose $\vdash \neg\, \gamma \to \pi$. Hence, $\vdash \sigma \to \gamma$. Since
      $\vdash \sigma \to \apr\sigma$, we find $\vdash \sigma \to \apr\gamma$. Ergo,
      $\vdash \sigma \to \neg\,\gamma$. We may conclude $\vdash \pi$.
  \end{proof}

  Thus, the sentences produced by Shavrukov {\&} Visser satisfy the following stronger kind of independence.

\begin{description}
\item[Double $\Pi^0_1$-conservativity] If ${\sf PA}+ \phi$ is consistent, then both 
${\sf PA}+ \phi+\rho(\gnum\phi)$ and ${\sf PA}+ \phi+\neg\, \rho(\gnum\phi)$ are $\Pi^0_1$-conservative over $\PA + \varphi$.
\end{description}

For a systematic study of the existence of such doubly conservative sentences, see a recent work of Kogure and Kurahashi \cite{kogu:doubly25}.
We close this section by proposing the following question. 

 \begin{question}
For each $n \geq 2$, does there exist a $\Pi^0_n$-formula $\rho(x)$ satisfying the following conditions?
\begin{description}
\item[$\Pi^0_n$-Strong Independence] If ${\sf PA}+ \phi$ is consistent, then ${\sf PA}+\phi +\rho(\gnum\phi)$ is 
consistent and ${\sf PA}+\phi+\neg \, \rho(\gnum \phi)$ is
$\Pi^0_n$-conservative over ${\sf PA}+\phi$.
\item[Conditional Extensionality] If ${\sf PA} \vdash \phi \iff \psi$, then 
$ {\sf PA}+ \phi \vdash\rho(\gnum \phi) \iff \rho(\gnum\psi)$.
\end{description}

Furthermore, for $n \geq 1$, can we strengthen Conditional Extensionality to Consistent Extensionality?
 \end{question}

  \section{There are no extensional Rosser Formulas}\label{nerf}
  In this section, we prove that there is no formula that is both independent and
  extensional, thus answering
   Hamkins' original question in the negative.
  
  \begin{theorem}\label{grotesmurf}
  Let $U$ be any consistent theory that extends the Tarski-Mostowski-Robinson theory {\sf R}.
  Then, there are no extensional Rosser formulas over base theory $U$.
  \end{theorem}
  
  \noindent
  Note that we do not need any constraint on the complexity of the axiom set of $U$.
  \begin{proof}
  We write $\vdash$ for $U$-provability. 
  Suppose $\rho(x)$ is an extensional Rosser formula over $U$. We form the fixed points
  $\vdash \phi_0 \iff \rho(\gnum {\phi_0})$ and $\vdash \phi_1 \iff \neg\, \rho(\gnum {\phi_1})$.
  By Independence, we find $\vdash \neg\,\phi_0$ and $\vdash \neg\,\phi_1$. 
  So, $\vdash \phi_0 \iff \phi_1$. 
  By Extensionality, $\vdash \rho(\gnum{\phi_0}) \iff \rho(\gnum{\phi_1})$.
  By the Fixed Point Equations, we find $\vdash \phi_0 \iff \neg \, \phi_1$, contradicting the consistency of $U$.  
  \end{proof}
  \noindent By minor adaptations of the formulation and the proof, we find a similar result for theories that \emph{interpret}
  {\sf R}. 
  
  The above argument only uses the special case of Extensionality for inconsistent formulas.
  So, one may wonder what happens if we simply ban this case. 
  This gives an insight into why our Theorem~\ref{con_ext} holds.

\section{There is a consistently extensional Rosser Formula}\label{csorf}
In this section, we prove the existence of a $\Pi^0_1$-formula that is a consistently extensional Rosser formula $\rho(x)$. 
In the light of Theorem \ref{grotesmurf}, this seems to provide the best positive answer to Hamkins' question.
Our construction of $\rho(x)$ is based on the method developed by Guaspari and Solovay \cite{guas:ross79}. 
Let $\varphi^0$ and $\varphi^1$ denote $\varphi$ and $\neg\, \varphi$, respectively. 

We will assume monotonicity of coding, to wit
that (the code of) a proof is, verifiably, larger than the codes of the formulas
occurring in the proof and that (the code of) a formula is, verifiably, larger than the codes of its strict subformulas.

\begin{theorem}\label{con_ext}
    Let $U$ be any consistent c.e.~extension of $\EA$. 
        Then, there exists a $\Pi^0_1$-formula $\rho(x)$ satisfying the following two conditions:
    \begin{description}
        \item [Independence] If $U + \varphi$ is consistent, then so are $U + \varphi + \rho(\gnum{\varphi})$ and $U + \varphi + \neg\, \rho(\gnum{\varphi})$. 

        \item [Consistent Extensionality] If $U + \varphi$ is consistent and $U \vdash \varphi \leftrightarrow \psi$, then $\EA \vdash \rho(\gnum{\varphi}) \leftrightarrow \rho(\gnum{\psi})$. 
    \end{description}
\end{theorem}

\begin{proof}
Let $U$ be any consistent c.e.~extension of $\EA$. 
We fix an elementary formula that is a proof predicate $\Proof_U(x, y, z)$ saying that $y$ is a $(U + x)$-proof of $z$. 
As mentioned in the previous section, the existence of such a formula is guaranteed by Craig’s trick. 
Let $\Prov_U(x, z)$ be the $\Sigma^0_1$-formula $\exists y\, \Proof_U(x, y, z)$ which is a provability predicate of $U + x$. 

In the following, we simultaneously define a 2-ary elementary function $f(\cdot, \cdot)$ and a 
sequence $\{\sim_k\}_k$ of equivalence relations whose underlying sets are $\verz{0,\dots,k-1}$.
For each sentence $\varphi$, the function $f(\varphi, \cdot)$ enumerates all theorems of $U + \varphi$. 
Let $\PR_f(x, z)$ be the $\Sigma^0_1$ formula $\exists y\, (f(x, y) = z)$. 
By the Fixed Point Theorem, a $\Pi^0_1$-formula $\rho(x)$ satisfying the following equivalence is effectively found: 
\[
    \EA \vdash \rho(x) \leftrightarrow \neg \left(\PR_f(x, \gnum{\rho(\dot{x})}) < \PR_f(x, \gnum{\neg\, \rho(\dot{x})}) \right).
\]
By using the parametrised Recursion Theorem in its formalised form over $\EA$, we may use the formulas $\PR_f(x, z)$ and $\rho(x)$ and the relations $\{\sim_k\}$ in the definition of $f$. 
Also, we may use the formula $\rho(x)$ in the definition of $\sim_k$. 
We would like to prove that the $\Pi^0_1$-formula $\rho(x)$ witnesses the statement of the theorem. 

First, we recursively define the sequence $\{\sim_k\}$ of equivalence relations by referring to $(U + \varphi)$-proofs based on $\Proof_U(\varphi, y, z)$ as follows. 
Let ${\sim_0} : = \emptyset$ viewed as an equivalence relation on the empty domain. 
Suppose that we have already defined the relation $\sim_k$ and we define the relation $\sim_{k+1}$ by distinguishing the following two cases: 
\begin{enumerate}
    \item [\textup(X\textup)] If $k+1$ is a $U$-proof of a formula of the form $\varphi \leftrightarrow \psi$ and for all 
    $\gamma$ with ($\varphi \sim_{k} \gamma$ or $\psi \sim_{k} \gamma$) and all $p \leq k$, $p$ is a 
    $(U + \gamma)$-proof of neither $\rho(\gnum{\gamma})$ nor $\neg\, \rho(\gnum{\gamma})$. 
    Then, let $\sim_{k+1}$ be the smallest equivalence relation extending $\sim_k \cup \{(\varphi, \psi)\}$ on $\verz{0,\dots,k}$.
    \item [\textup(Y\textup)] Otherwise, let $\sim_{k+1}$ be the  smallest equivalence relation extending $\sim_k$ on $\verz{0,\dots,k}$.
\end{enumerate}

It is easy to see that $\EA \vdash \forall y, z \, (z \geq y \to$ `$\sim_{z}$ extends $\sim_{y}$'$)$. 

Next, we define the function $f(\cdot, \cdot)$. 
For each sentence $\varphi$, the definition of $f(\varphi, \cdot)$ consists of Procedures 1 and 2. 
It starts with Procedure 1 where $f(\varphi, \cdot)$ outputs formulas in stages referring to $(U + \varphi)$-proofs. 
In Stage $k$, we define the value of $f(\varphi, k)$.  
A $\varphi$-bell is prepared and may ring during Procedure 1. 
After the $\varphi$-bell rings, the definition of $f(\varphi, k)$ switches to Procedure 2 where $f(\varphi, \cdot)$ outputs all formulas. 
Such a definition of functions preparing a bell originates Guaspari and Solovay \cite{guas:ross79}. 

\medskip
\textsc{Stage $k$}: Let $\varphi$ be any sentence.
We define the value of $f(\varphi, k)$. 
We call this stage \textsc{Stage $\varphi$-$k$}.

\medskip
\textsc{Procedure 1}: The $\varphi$-bell has not yet rung. 

We distinguish the following three cases: 

\begin{enumerate}
    \item [\textup{(A)}] If $k$ is not a $(U + \varphi)$-proof of any formula, then $f(\varphi, k)$ is $0$. 

\medskip
    Let $i \in \{0, 1\}$. 

    \item [\textup{(B$_i$)}] If $k$ is a $(U + \varphi)$-proof of $\rho(\gnum{\varphi})^i$ and there exist a number $p < k$ and a sentence $\gamma$ satisfying the following three conditions:
        \begin{enumerate}[1.]
            \item $\varphi \sim_p \gamma$, 
            \item $p$ is a $(U + \gamma)$-proof of $\rho(\gnum{\gamma})^{1-i}$, 
            \item for any $q < p$ and sentence $\lambda$ with $\varphi \sim_p \lambda$, we have that $q$ is a $(U + \lambda)$-proof of neither $\rho(\gnum{\lambda})$ nor $\neg\, \rho(\gnum{\lambda})$. 
         \end{enumerate}
        Then, let $f(\varphi, k) = \rho(\gnum{\varphi})^{1-i}$. 
        Ring the $\varphi$-bell. 

   \medskip
   
    \item [\textup{(C)}] Else, if $k$ is a $(U + \varphi)$-proof of a formula $\sigma$, then $f(\varphi, k) = \sigma$. 
\end{enumerate}

We have completed the definition of \textsc{Stage $\varphi$-$k$}. 

    \medskip

    \textsc{Procedure 2}: The $\varphi$-bell have already rung at \textsc{Stage $\varphi$-$m$} for $m < k$.\\ 
    Let $\{\xi_i\}$ be an effective enumeration of all formulas and let $f(\varphi, m+ 1+i) = \xi_i$. 

    \medskip
We have finished the definition of $f$. 

    \medskip
In the following, we shall prove several claims.
Let $\Bell(x, y)$ and $\Sent(x)$ be elementary formulas saying that `the $x$-bell rings at \textsc{Stage $x$-$y$}' and `$x$ is a sentence', respectively. 

\begin{claim}\label{cl:bell}
$\EA \vdash \forall x \, \bigl(\Sent(x) \land \exists y \, \Bell(x, y) \to \neg \, \Con(U + x) \bigr)$, where $\Con(U + x) \equiv \neg \, \Prov_U(x, \gnum{\bot})$. 
\end{claim}
\begin{proof}
We argue in $\EA$. 
Let $\varphi$ be any sentence. 
Suppose that the $\varphi$-bell rings at \textsc{Stage $\varphi$-$k$} based on (B$_i$) for $i \in \{0, 1\}$. 
In this case, $k$ is a $(U + \varphi)$-proof of $\rho(\gnum{\varphi})^i$ and there exist a number $p < k$ and a sentence $\gamma$ satisfying the following three conditions:
        \begin{enumerate}[1.]
            \item $\varphi \sim_p \gamma$, 
            \item $p$ is a $(U + \gamma)$-proof of $\rho(\gnum{\gamma})^{1-i}$, 
            \item for any $q < p$ and sentence $\lambda$ with $\varphi \sim_p \lambda$, we have that $q$ is a $(U + \lambda)$-proof 
            of neither $\rho(\gnum{\lambda})$ nor $\neg\, \rho(\gnum{\lambda})$. 
        \end{enumerate}
        By (1) and (3), each $q < p$ is a $(U + \gamma)$-proof of neither $\rho(\gnum{\gamma})$ nor $\neg\, \rho(\gnum{\gamma})$. 
        So, $\rho(\gnum{\gamma}), \neg\, \rho(\gnum{\gamma}) \notin \{f(\gamma, 0), \ldots, f(\gamma, p-1)\}$ and the $\gamma$-bell has 
        not yet rung before \textsc{Stage} $\gamma$-$p$. 

        Suppose, towards a contradiction, that \textsc{Stage $\gamma$-$p$} of $f(\gamma, \cdot)$ matches (B$_{1-i}$). 
        Then, there exist $r < p$ and $\delta$ such that $\gamma \sim_r \delta$ and $r$ is a $(U + \delta)$-proof of $\rho(\gnum{\delta})^i$. 
        Since $\varphi \sim_p \gamma \sim_r \delta$ and $r < p$, we have $\varphi \sim_p \delta$. 
        This violates (3). 
        
        So, \textsc{Stage $\gamma$-$p$} of $f(\gamma, \cdot)$ does not match (B$_{1-i}$) and hence $f(\gamma, p) = \rho(\gnum{\gamma})^{1-i}$ based on (C). 
        Therefore, 
        \[
            \PR_f(\gnum{\gamma}, \gnum{\rho(\gnum{\gamma})^{1-i}}) < \PR_f(\gnum{\gamma}, \gnum{\rho(\gnum{\gamma})^i})
        \]
        holds.  
        By formalised $\Sigma^0_1$-completeness, this sentence is $(U + \gamma)$-provable.
        By the witness comparison argument, this sentence implies $\rho(\gnum{\gamma})^i$, so it is also $(U + \gamma)$-provable. 
        Thus $U + \gamma$ is inconsistent because $\rho(\gnum{\gamma})^{1-i}$ also has a $(U + \gamma)$-proof. 
        Since $\varphi \sim_p \gamma$, we conclude that $U + \varphi$ is inconsistent. 
\end{proof}

\begin{claim}\label{cl:equiv}
$\EA \vdash \forall x \, \bigl(\Sent(x) \to \forall y \left(\Prov_{U}(x, y) \leftrightarrow \PR_f(x, y) \right) \bigr)$. 
\end{claim}
\begin{proof}
By the definition of $f$, 
\[
    \EA \vdash \Sent(x) \land \neg \, \exists y \, \Bell(x, y) \to \forall y \left(\Prov_{U}(x, y) \leftrightarrow \PR_f(x, y) \right).  
\]
Since $f$ outputs all formulas in Procedure 2, 
\[
    \EA \vdash \Sent(x) \land \exists y \, \Bell(x, y) \to \forall y \left(\Fml(y) \leftrightarrow \PR_f(x, y) \right), 
\]
where $\Fml(x)$ is an elementary formula saying that $x$ is a formula. 
It is easy to see $\EA \vdash \Sent(x) \land \neg \, \Con(U + x) \to \forall y \left(\Fml(y) \leftrightarrow \Prov_{U}(x, y) \right)$. 
By combining these equivalences with Claim \ref{cl:bell}, 
\[
    \EA \vdash \Sent(x) \land \exists y \, \Bell(x, y) \to \forall y \left(\Prov_{U}(x, y) \leftrightarrow \PR_f(x, y) \right). 
\]
By the law of excluded middle, we conclude: \qedright
\[\EA \vdash \forall x \, \bigl(\Sent(x) \to \forall y \left(\Prov_{U}(x, y) \leftrightarrow \PR_f(x, y) \right) \bigr).\] 
\end{proof}

\begin{claim}\label{cl:Rosser}
If $U + \varphi$ is consistent, then so are $U + \varphi + \rho(\gnum{\varphi})$ and $U + \varphi + \neg\, \rho(\gnum{\varphi})$. 
\end{claim}
\begin{proof}
    Suppose that $U + \varphi$ is consistent. 
    By the soundness of $\EA$, the sentence as in Claim \ref{cl:equiv} is true, and thus $\PR_f(\gnum{\varphi}, y)$ is a provability predicate of $U + \varphi$. 
    So, the $\Pi^0_1$ sentence $\rho(\gnum{\varphi})$ is a $\Pi^0_1$-Rosser sentence of $U + \varphi$. 
    The claim follows by the proof of Rosser's theorem.  
\end{proof}

\begin{claim}\label{cl:bound}
Let $\varphi$ be any sentence such that $U + \varphi$ is consistent. 
Then $\EA$ proves the following statement.
For any sentence $\gamma$ and any numbers $k$ and $p$, if $\varphi \sim_k \gamma$ and $p$ is a $(U + \gamma)$-proof 
of one of $\rho(\gnum{\gamma})$ and $\neg\, \rho(\gnum{\gamma})$, then there exists an $l < p$ such that $\varphi \sim_l \gamma$.
\end{claim}
\begin{proof}
We reason in $\EA$. 
Suppose $\varphi \sim_k \gamma$ and $p$ is a $(U + \gamma)$-proof of one of $\rho(\gnum{\gamma})$ and $\neg\, \rho(\gnum{\gamma})$. 
Let $k_0$ be the smallest number such that $\varphi \sim_{k_0+1} \gamma$ (note that ${\sim_0} = \emptyset$). 
If $\gamma$ is $\varphi$ itself, then $p$ is non-standard by Claim \ref{cl:Rosser}. 
For the standard G\"odel number $m$ of $\varphi$, we have $m+1 < p$ and $\varphi \sim_{m+1} \gamma$. 
We have that $l : = m+1$ satisfies the requirements. 

So, we may assume that $\gamma$ is distinct from $\varphi$. 
By the choice of $k_0+1$, we have $\varphi \not \sim_{k_0} \gamma$, and the definition of $\sim_{k_0+1}$ matches (X). 
Then, there exist sentences $\delta_0$ and $\delta_1$ and $i \in \{0, 1\}$ such that $k_0+1$ is a proof of 
$\delta_0 \leftrightarrow \delta_1$, $\varphi \sim_{k_0} \delta_{1-i}$ and $\gamma \sim_{k_0} \delta_i$. 

Suppose, towards a contradiction, that $p \leq k_0$. 
Since $\gamma \sim_{k_0} \delta_i$ and $p$ is a $(U + \gamma)$-proof of $\rho(\gnum{\gamma})$ or 
$\neg\, \rho(\gnum{\gamma})$, we have that the definition of $\sim_{k_0+1}$ does not match (X), a contradiction. 
So, we have $k_0+1 \leq p$. 
Since $\rho(\gnum{\gamma})$ differs from $\delta_0 \leftrightarrow \delta_1$, we have $k_0+1 \neq p$. 
Therefore, $k_0+1 < p$. 
So, $l : = k_0+1$ satisfies the requirements. 
\end{proof}

Claim \ref{cl:Rosser} together with the following Claim \ref{cl:e} completes the proof of the theorem.

\begin{claim}\label{cl:e}
If $U + \varphi$ is consistent and $U \vdash \varphi \leftrightarrow \psi$, then $\EA \vdash \rho(\gnum{\varphi}) \leftrightarrow \rho(\gnum{\psi})$. 
\end{claim}
\begin{proof}
Suppose that $U + \varphi$ is consistent and $U \vdash \varphi \leftrightarrow \psi$. 
Then, $U + \psi$ is also consistent. 
Let $n$ be a (standard) $U$-proof of $\varphi \leftrightarrow \psi$. 
It suffices to prove that $\EA \vdash \neg\, \rho(\gnum{\varphi}) \to \neg\, \rho(\gnum{\psi})$. 

We reason in $\EA$: 
Suppose that $\neg\, \rho(\gnum{\varphi})$ holds, that is, 
\[
    \PR_f(\gnum{\varphi}, \gnum{\rho(\gnum{\varphi})}) < \PR_f(\gnum{\varphi}, \gnum{\neg\, \rho(\gnum{\varphi})}) \tag{$\ast$}
\]
holds. 
Then, $\PR_f(\gnum{\varphi}, \gnum{\rho(\gnum{\varphi})})$ holds, that is, $f(\varphi, \cdot)$ eventually outputs $\rho(\gnum{\varphi})$. 
By Claim \ref{cl:equiv}, $U + \varphi$ proves $\rho(\gnum{\varphi})$. 
By formalised $\Sigma^0_1$-completeness, $U + \varphi$ proves the true $\Sigma^0_1$-sentence $\neg\, \rho(\gnum{\varphi})$. 
Hence $U + \varphi$ is inconsistent. 
Since $U \vdash \varphi \leftrightarrow \psi$, we also have that $U + \psi$ is inconsistent. 

Let $k$ be the smallest $(U + \varphi)$-proof of either $\rho(\gnum{\varphi})$ or $\neg\, \rho(\gnum{\varphi})$. 
Let $l$ be the smallest $(U + \psi)$-proof of either $\rho(\gnum{\psi})$ or $\neg\, \rho(\gnum{\psi})$. 
By Claim \ref{cl:Rosser} and the fact that $n$ is standard, we have $k > n$ and $l > n$. 
It follows that $\varphi \sim_n \psi$. 
Also, by construction, $f(\varphi, \cdot)$ outputs neither $\rho(\gnum{\varphi})$ nor 
$\neg\, \rho(\gnum{\varphi})$ before \textsc{Stage $\varphi$-$k$}. 
So, $(\ast)$ implies $f(\varphi, k) = \rho(\gnum{\varphi})$. 
We would like to show that $\neg\, \rho(\gnum{\psi})$ holds. 
By the choice of $l$, it suffices to prove $f(\psi, l) = \rho(\gnum{\psi})$. 
We prove $f(\psi, l) = \rho(\gnum{\psi})$ by distinguishing several cases.

\medskip
Case 1: $k$ is a $(U + \varphi)$-proof of $\neg\, \rho(\gnum{\varphi})$. \\
Since $f(\varphi, k) = \rho(\gnum{\varphi})$, this output is based on (B$_1$) at \textsc{Stage $\varphi$-$k$}. 
Thus, there exist a number $p < k$ and a sentence $\gamma$ satisfying the following three conditions:
        \begin{enumerate}[i.]
            \item  $\varphi \sim_p \gamma$, 
            \item $p$ is a $(U + \gamma)$-proof of $\rho(\gnum{\gamma})$, 
            \item  for any $q < p$ and sentence $\lambda$ with $\varphi \sim_p \lambda$, we have that $q$ is a $(U + \lambda)$-proof of 
            neither $\rho(\gnum{\lambda})$ nor $\neg\, \rho(\gnum{\lambda})$. 
         \end{enumerate}
By Claim \ref{cl:Rosser}, we have $p > n$. 
So, $\varphi \sim_p \psi$. 

\medskip

Case 1.1: $l$ is a $(U + \psi)$-proof of $\rho(\gnum{\psi})$. \\
Suppose, towards a contradiction, that the construction of $f(\psi, l)$ matches the condition (B$_0$) at \textsc{Stage $\psi$-$l$}. 
Then, there exist a number $r < l$ and a sentence $\delta$ satisfying the following three conditions:
        \begin{enumerate}[i.]
        \setcounter{enumi}{3}
            \item  $\psi \sim_r \delta$, 
            \item $r$ is a $(U + \delta)$-proof of $\neg\, \rho(\gnum{\delta})$, 
            \item  for any $q < r$ and sentence $\lambda$ with $\psi \sim_r \lambda$, we have that $q$ is a $(U + \lambda)$-proof of 
            neither $\rho(\gnum{\lambda})$ nor $\neg\, \rho(\gnum{\lambda})$. 
         \end{enumerate}

Since $\rho(\gnum{\gamma})$ differs from $\neg\, \rho(\gnum{\delta})$, we have $p \neq r$. 
If $r < p$, then we have $\varphi \sim_p \psi \sim_p \delta$ and $r$ is a $(U + \delta)$-proof of $\neg\, \rho(\gnum{\delta})$.
This violates (iii). 
If $p < r$, then we have $\psi \sim_r \varphi \sim_r \gamma$ and $\rho(\gnum{\gamma})$ has a $(U + \gamma)$-proof $p < r$.
This violates (vi). 

In either case, we have a contradiction. 
Hence, the construction of $f(\psi, l)$ does not match the condition (B$_0$) at \textsc{Stage $\psi$-$l$}. 
We conclude that $f(\psi, l) = \rho(\gnum{\psi})$ based on (C).  

\medskip

Case 1.2: $l$ is a $(U + \psi)$-proof of $\neg\, \rho(\gnum{\psi})$. \\
Since $\rho(\gnum{\gamma})$ differs from $\neg\, \rho(\gnum{\psi})$, we have $p \neq l$. 
Since $\varphi \sim_p \psi$ and $l$ is a $(U + \psi)$-proof of $\neg\, \rho(\gnum{\psi})$, the assumption $l < p$ violates (iii). 
So, we have $p < l$. 
Since $\psi \sim_p \varphi \sim_p \gamma$, we have that $p$ and $\gamma$ witness the condition (B$_1$) at \textsc{Stage $\psi$-$l$}. 
Therefore, we conclude $f(\psi, l) = \rho(\gnum{\psi})$. 

\medskip
Case 2: $k$ is a $(U + \varphi)$-proof of $\rho(\gnum{\varphi})$. 

\medskip
Case 2.1: $l$ is a $(U + \psi)$-proof of $\rho(\gnum{\psi})$. \\
Suppose, towards a contradiction, that the construction of $f(\psi, l)$ matches the condition (B$_0$) at \textsc{Stage $\psi$-$l$}. 
Then, there exist a number $r < l$ and a sentence $\delta$ satisfying the following three conditions:
        \begin{enumerate}[i.]
        \setcounter{enumi}{6}
            \item  $\psi \sim_r \delta$, 
            \item  $r$ is a $(U + \delta)$-proof of $\neg\, \rho(\gnum{\delta})$, 
            \item  for any $q < r$ and sentence $\lambda$ with $\psi \sim_r \lambda$, we have that $q$ is a $(U + \lambda)$-proof of 
            neither $\rho(\gnum{\lambda})$ nor $\neg\, \rho(\gnum{\lambda})$. 
         \end{enumerate}

Since $\rho(\gnum{\varphi})$ differs from $\neg\, \rho(\gnum{\delta})$, we have $k \neq r$. 
If $k < r$, then we have $\psi \sim_r \varphi$ and $k$ is a $(U + \varphi)$-proof of $\rho(\gnum{\varphi})$. 
This violates (ix). 
So, we have $r < k$. 
Since $\varphi \sim_r \psi \sim_r \delta$, we have that $r$ and $\delta$ witness the condition (B$_0$) at \textsc{Stage $\varphi$-$k$}. 
This implies $f(\varphi, k) = \neg\, \rho(\gnum{\varphi})$, a contradiction. 

Hence, the construction of $f(\psi, l)$ does not match the condition (B$_0$) at \textsc{Stage $\psi$-$l$}. 
We conclude that $f(\psi, l) = \rho(\gnum{\psi})$ based on (C).  

\medskip
Case 2.2: $l$ is a $(U + \psi)$-proof of $\neg\, \rho(\gnum{\psi})$. \\
Since $\rho(\gnum{\varphi})$ differs from $\neg\, \rho(\gnum{\psi})$, we have $k \neq l$. 

\medskip
Case 2.2.1: $l < k$.\\
We have $\psi \sim_k \varphi$ and $k$ is a $(U + \varphi)$-proof of $\rho(\gnum{\varphi})$. 
By the least number principle, we find a number $r \leq k$ and a sentence $\delta$ satisfying the following conditions:
\begin{enumerate}[i.]
\setcounter{enumi}{9}
    \item  $\psi \sim_r \delta$,
    \item  $r$ is a $(U + \delta)$-proof of either $\rho(\gnum{\delta})$ or $\neg\, \rho(\gnum{\delta})$, 
    \item  for any $q < r$ and sentence $\lambda$ with $\psi \sim_q \lambda$, we have that $q$ is a $(U + \lambda)$-proof of 
    neither $\rho(\gnum{\lambda})$ nor $\neg\, \rho(\gnum{\lambda})$. 
\end{enumerate}

By Claim \ref{cl:bound}, Clause (xii) can be strengthened as follows: 
\begin{enumerate}[i$'$.]
\setcounter{enumi}{11}
    \item  for any $q < r$ and sentence $\lambda$ with $\psi \sim_r \lambda$, we have that $q$ is a $(U + \lambda)$-proof of 
    neither $\rho(\gnum{\lambda})$ nor $\neg\, \rho(\gnum{\lambda})$. 
\end{enumerate}

Suppose, towards a contradiction, that $r$ is a $(U + \delta)$-proof of $\neg\, \rho(\gnum{\delta})$. 
Since $\neg\, \rho(\gnum{\delta})$ differs from $\rho(\gnum{\varphi})$, we have $r \neq k$. 
Then $r < k$. 
Since $\varphi \sim_r \psi \sim_r \delta$, (x) and (xii$'$) imply that $r$ and $\delta$ witness the condition (B$_0$) at \textsc{Stage $\varphi$-$k$}. 
Thus, $f(\varphi, k) = \neg\, \rho(\gnum{\varphi})$, a contradiction. 

Therefore, (xi) implies that $r$ is a $(U + \delta)$-proof of $\rho(\gnum{\delta})$. 
If $l \leq r$, then we have $l < r$ because $\neg\, \rho(\gnum{\psi})$ differs from $\rho(\gnum{\delta})$. 
Since $\psi \sim_r \psi$ and $l$ is a $(U + \psi)$-proof of $\neg\, \rho(\gnum{\psi})$, this violates (xii$'$). 
So, we have $r < l$.
Thus, (x) and (xii$'$) imply that $r$ and $\delta$ witness the condition (B$_1$) at \textsc{Stage $\psi$-$l$}. 
We conclude $f(\psi, l) = \rho(\gnum{\psi})$. 

\medskip 

Case 2.2.2: $k < l$.\\
Since $\varphi \sim_l \psi$ and $l$ is a $(U + \psi)$-proof of $\neg\, \rho(\gnum{\psi})$, by the least number principle, we find a 
number $p \leq l$ and a sentence $\gamma$ satisfying the following conditions:
\begin{enumerate}[i.]
\setcounter{enumi}{12}
    \item  $\varphi \sim_p \gamma$,
    \item  $p$ is a $(U + \gamma)$-proof of either $\rho(\gnum{\gamma})$ or $\neg\, \rho(\gnum{\gamma})$, 
    \item  for any $q < p$ and sentence $\lambda$ with $\varphi \sim_q \lambda$, we have that $q$ is a $(U + \lambda)$-proof of 
    neither $\rho(\gnum{\lambda})$ nor $\neg\, \rho(\gnum{\lambda})$. 
\end{enumerate}

By Claim \ref{cl:bound}, Clause (xv) can be strengthened as follows: 
\begin{enumerate}[i$'$.]\setcounter{enumi}{14}
    \item for any $q < p$ and sentence $\lambda$ with $\varphi \sim_p \lambda$, we have that $q$ is a $(U + \lambda)$-proof of 
    neither $\rho(\gnum{\lambda})$ nor $\neg\, \rho(\gnum{\lambda})$. 
\end{enumerate}

Suppose, towards a contradiction, that $p$ is a $(U + \gamma)$-proof of $\neg\, \rho(\gnum{\gamma})$. 
If $k \leq p$, then we have $k < p$ because $\rho(\gnum{\varphi})$ differs from $\neg\, \rho(\gnum{\gamma})$. 
Since $\varphi \sim_p \varphi$ and $p$ is a $(U + \varphi)$-proof of $\rho(\gnum{\varphi})$, this violates (xv$'$). 
So, we have $p < k$.
Then, (xiii) and (xv$'$) imply that $p$ and $\gamma$ witness the condition (B$_0$) at \textsc{Stage $\varphi$-$k$}. 
Thus, $f(\varphi, k) = \neg\, \rho(\gnum{\varphi})$, a contradiction. 

Therefore, (xiv) implies that $p$ is a $(U + \gamma)$-proof of $\rho(\gnum{\gamma})$. 
Since $\neg\, \rho(\gnum{\psi})$ differs from $\rho(\gnum{\gamma})$, we have $l \neq p$, and hence $p < l$. 
Since $\psi \sim_p \varphi \sim_p \gamma$, we have that $p$ and $\gamma$ witness the condition (B$_1$) at \textsc{Stage $\psi$-$l$}. 
We conclude $f(\psi, l) = \rho(\gnum{\psi})$. 
\end{proof}
We have finished our proof of the theorem. 
\end{proof}

\section{A conditionally extensional \strynd\ Formula}\label{cdorf}
In this section, we develop our second positive result, to wit Theorem~\ref{maintheorem}. There is a conditionally extensional \strynd\ formula for
  c.e.~base theories extending {\sf PA}. We note that our previous positive result, Theorem~\ref{con_ext}
  works for extensions of {\sf EA}. Moreover, it only requires consistent extensionality. 
  
  We think that, using the tricks from \cite{viss:abso21}, one can reduce the assumption that we have an extension
  of {\sf PA} to the assumption of weaker theories, but we did not explore this idea further. In any case, the presentation for extensions of {\sf PA} is far simpler.
  Also, the result can be strengthened to c.e. sequential theories with induction for the full language on some interpretation of {\sf Q}.
  This last improvement requires minor adaptations.

  What we gain in the present result is strong independence, that is the negation of the independent sentence is $\Pi_1$-conservative over the base
  theory.
  
The conditionally extensional \strynd\ $\Pi^0_1$-formula offered in this paper is an adaptation of the construction of 
  \cite{viss:abso21}. It  delivers slow consistency statements in the style of
 \cite{viss:abso21}. See also \cite{frie:slow13} for a more proof-theoretic take on 
 slow provability. The construction of the formulas is fixed-point-free.  
 
 \subsection{Conventions and basic Definitions}

  We will assume we use a coding like the Smullyan coding. See \cite{viss:numb25} for a careful presentation of 
  that coding. The main property we will use is that the code of a syntactic object of length  (= number of symbols) $n$ is of the order
  $2^{cn}$ for standard $c$ and that the arithmetical consequences of this fact are {\sf EA}-verifiable. For example,
  {\sf EA} will prove that the tracking function of concatenation is of the order of multiplication.
  As in Section~\ref{csorf}, we will also assume that our coding is monotonic, i.e., the code of a syntactic object (term, formula, proof) that is part of another syntactic object 
  is smaller than the code of the syntactic object in which it occurs.

  We will need a somewhat more refined analysis of the proof and provability. The formula $\Bews(x,y)$ stands for
  some standard elementary arithmetisation of `$x$ is a proof  of $y$ from assumptions'.
  Let the elementary formula $\Ass(x,y)$ stand for `$y$ is an assumption that occurs in the proof $x$'. 
  As usual, we assume that $\Ass(x,y)$ implies $y<x$ over
  {\sf EA}. We define:
  \begin{itemize}
  \item
   $\Proof_{\tupel\alpha}(x,y) :\iff \Bews(x,y) \wedge \forall z < x\, (\Ass(x,z) \to \alpha(z))$.
   \item
    $\opr_{\tupel\alpha}\phi :\iff \exists x\, \Proof_{\tupel\alpha}(x,\gnum \phi)$. 
    \item
    As usual we write $\oco$ for $\neg\opr\neg$.
  \item
   $\opr_{\tupel{\alpha}+\phi}\psi :\iff \opr_{\tupel{\alpha'}}\psi$, where $\alpha'(x) \iff (\alpha(x) \vee x= \gnum\phi)$.\\
   If $\alpha$ is contextually given we will usually omit it, writing just $\opr_\phi\psi$.
   \end{itemize}
   
   \noindent
   Our meta-theory in this section will be \eac. Here $\mathrm B\Sigma_1$ is $\Sigma_1$-collection. The advantage of having collection is that
we can work with the $\Sigma_1$-sentences are, modulo provable equivalence, closed under universal bounded quantification.

  \subsection{The $\Pi^0_1$-conservativity of Inconsistency}
We will need the well-known fact that inconsistency statements are $\Pi^0_1$-conservative.
Per Lindstr\"om, in \cite[p.~94]{lind:aspe03}, ascribes this result to Georg Kreisel in \cite{kreis:weak62}. 
Since, over Peano Arithmetic, $\Pi^0_1$-conservativity coincides with interpretability, the result also
follows from Feferman's result concerning the interpretability of inconsistency. See \cite{fefe:arit60}.
Feferman's original proof only works for extensions of {\sf PA}. Per Lindstr\"om and Albert Visser
discovered, independently, a simpler proof that works for a wider range of theories.

Here we give a direct proof for the case of $\Pi^0_1$-conservativity.

\begin{theorem}\label{kreiselsmurf}
    Let $V$ be an extension of  \eac\ and let the axiom set
    of $V$ be given by a $\Sigma_1$-predicate $\alpha$. Then, the inconsistency statement $\opr_{\,\tupel{\alpha}} \bot$
    is $\Pi^0_1$-conservative over $V$.
\end{theorem}

\begin{proof}
We note that $\opr_{\,\tupel{\alpha}}$ is (\eac)-provably equivalent to a $\Sigma^0_1$-formula, so it will, over \eac,
satisfy the third L\"ob Condition, to wit, provable implies provably provable.

    Consider any $\Pi^0_1$-sentence $\pi$ and let $\sigma$ be a $\Sigma^0_1$-sentence that is (\eac)-provably equivalent
    to $\neg\, \pi$. Suppose $V+\opr_{\,\tupel{\alpha}} \bot \vdash \pi$. Then, $V+\sigma \vdash \oco_{\,\tupel{\alpha}}\top$ and, hence,
    $V+\sigma \vdash \oco_{\,\tupel{\alpha}}\sigma$. So, by the Second Incompleteness Theorem 
    $V+\sigma \vdash \bot$. Ergo, $V\vdash \pi$.
\end{proof}

Curiously, even if $\Pi^0_1$-conservativity and interpretability diverge on \eac\ and many of its extensions,
both the interpretability of inconsistency and the $\Pi^0_1$-conservativity of inconsistency remain valid.
This can be partially explained by the fact that we happen to have both
the $\Pi^0_1$-conservativity of $V$ over \eac\ plus the consistency
statement for $V$ and the interpretability of $V$ in \eac\ plus the consistency
statement for $V$.

\begin{remark}
We note that the negation $\neg \, \rho$ of $\Pi^0_1$-Rosser sentence $\rho$ is not $\Pi^0_1$-conservative since it implies $\neg\, (\Prov_V(\gnum{\neg\, \rho}) \leq \Prov_V(\gnum{\rho}))$.
\end{remark}
  
\subsection{Uniform Provability}\label{prinsbsmurf}
We develop the basics of uniform provability  in this subsection.
We fix our base theory $U$ that, (\eac)-verifiably, 
extends {\sf PA} (in the language of {\sf PA}). 
It follows that $U$ is, (\eac)-verifiably, essentially reflexive. The axiomatisation of $U$ will be given by
the elementary formula $\alpha$. We will suppress the subscript $\tupel\alpha$, thus writing e.g.
$\opr\psi$ for $\opr_{\tupel{\alpha}}\psi$ and $\opr_\phi\psi$ for $\opr_{\tupel{\alpha}+\phi}\psi$, etcetera.

 Uniform proofs are a minor modification of proofs. We will use uniform provability to define
 our notion of smallness and we will use smallness to define slow provability as provability from
 small axioms.
 
We define the relation $\preceq_n$ on sentences $\leq n$ as follows. We have:
$\phi \preceq_n \psi$ iff $\psi$ follows by propositional logic from $\phi$ plus the sentences 
provable in $U$ by proofs $\leq n$. 
We use $\preceq_x$ for the arithmetisation of $\preceq_n$.
We note that $\preceq_x$ will be (\eac)-provably reflexive and transitive.

 We define:
 \begin{itemize}
  \item
   $\opr_{\phi,(x)}\psi :\iff \exists p\bleq x\, \Proof_\phi(p,\gnum\psi)$,
\item
$\gopr_{\phi,(x)}\psi :\iff \exists \chi \bleq x \, (\phi \preceq_x \chi \wedge  \opr_{\chi,(x)} \psi)$,
\item
$\gopr_\phi\psi :\iff \exists x\, \gopr_{\phi,(x)} \psi$.
  \end{itemize}
  
  The grey box stands for \emph{uniform provability}. We note that the uniformity is w.r.t.~the formula in the subscript.
  The next theorem explicates the relationship between uniform provability and provability.
  
  \begin{theorem}\label{mageresmurf}
  \begin{enumerate}[1.]
  \item
  $\eac \vdash \opr_{\phi,(x)} \psi \to \gopr_{\phi,(x)}\psi$, 
  \item
  There is an elementary function $F$ such that
  $\eac \vdash \gopr_{\phi,(x)} \psi \to \opr_{\phi,(Fx)}\psi$, 
  \item
  $\eac \vdash \opr_\phi\psi \iff \gopr_\phi\psi$.
  \end{enumerate}
  \end{theorem}
  
  \begin{proof}
      We just sketch the proof of (2). We reason informally, but it will be clear that our argument can be
      formalised  in \eac, and that a standard $F$ can be read off from the proof. Let $x$ be given. 
      We will assume that $\phi$ is smaller than $x$, since $\phi$ is fixed and standard (otherwise, we replace $x$ in the argument by
      $\phi$). Suppose
      $\phi \preceq_x\chi$ and that $p\leq x$ is a proof of $\psi$ in $U$ plus $\chi$. 
      
      Let $\mc P$ be the set of $U$-proofs that are $\leq x$ and let $\mc C$ be the set of conclusions of $\mc P$.
      We know that there are $\leq x$ elements of $\mc C$ and that they are all $\leq x$ and, similarly, for $\mc P$.
      We can produce a proof $q$ (using only propositional logic) of $\chi$ from $\phi$ plus $\mc C$.
      We note that all sentences that play the role of atoms in our propositional proof are $\leq x$. It follows that
      the number of symbols in the proof is of the order $x2^x$. See e.g. \cite{kraj:boun95} for estimates on propositional proofs. 
      So, $q$ can be estimated by $2^{c2^x}$. Our desired proof, say $r$, of $\psi$ from $\phi$ can be obtained roughly by concatenating
      all the elements of $\mc P$ and $q$. So, $r$ can be estimated by $x^x2^{x2^x}$.
  \end{proof}
  
  The next theorem gives us the desired uniformity property.
  
  \begin{theorem}\label{bravesmurf}
  \begin{enumerate}[i.]
  \item
  Suppose $p$ is the code of a $U$-proof of $\phi \to \psi$.
  Then, \[\eac\vdash \forall x\bgeq \num p\,\forall \chi \bleq x\,(\gopr_{\psi,(x)}\chi \to \gopr_{\phi,(x)}\chi).\]
  \item
    Suppose $p$ is the code of a $U$-proof of $\phi \iff \psi$.
  Then, \[\eac\vdash \forall x\bgeq \num p\,\forall \chi \bleq x\,(\gopr_{\psi,(x)}\chi \iff \gopr_{\phi,(x)}\chi).\]
  \end{enumerate}
  \end{theorem}
 
\begin{proof}
We prove (i). Case (ii) is similar. Suppose $p$ is the code of a $U$-proof of $\phi \to \psi$.
It follows that $\phi \preceq_p \psi$ and, hence, that $\eac \vdash \forall x\bgeq \num p\; \phi \preceq_x \psi$.

We reason in \eac. Suppose $x\geq \num p$ and $\gopr_{\psi,(x)}\chi$.
So, for some $\theta$, we have $ \theta \leq x$ and  $\psi \preceq_x \theta$ and $\opr_{\theta,(x)} \chi$.
We also have $\phi \preceq_x \psi$ and, hence, $\phi \preceq_x\theta$. It follows that
 $\gopr_{\phi,(x)}\chi$
\end{proof}
 
 \subsection{Uniform Smallness}
 Let $U$ be our base theory as introduced in Section~\ref{prinsbsmurf}.
 
We use $\sigma$, $\sigma'$, \dots\ to range over (G\"odel numbers of) $\Sigma^0_1$-sentences. 
Let $\phi$ be any arithmetical sentence. 
We define \emph{uniform $\phi$-smallness} as follows:
\begin{itemize}
\item
$ \mf S_\phi(x) : \iff \forall \sigma\bleq x\, (\gopr_{\phi,(x)} \sigma \to \True_{\Sigma^0_1}(\sigma))$.
\end{itemize}

It is easily seen that, modulo (\eac)-provable equivalence, uniform $\phi$-smallness is $\Sigma^0_1$.
We will suppress the `uniform' in the rest of this paper, since we only consider the uniform case. 
  
Trivially, $\phi$-smallness is downwards closed.

If $\phi$ is consistent with $U$, then it is consistent with $U+\phi$ that not all numbers 
are small, since $U+\phi$ does not prove $\Sigma^0_1$-reflection for
$(U+\phi)$-provability. On the other hand, we have:

\begin{theorem}\label{kleinesmurf}
For every $n$, the theory $U+\phi$ proves that $n$ is $\phi$-small, i.e., 
$U+\phi\vdash \mf S_\phi(\underline n)$. Moreover, \eac\ verifies this insight, i.e.,
 $\eac \vdash \forall x\, \opr_\phi \mf S_\phi(x)$.
\end{theorem}

The principle articulated in the formalised part of the theorem is a typical example of an \emph{outside-big-inside-small principle}. 
Objects that may be  very big in the outer world are  small in the inner world.

We remind the reader that the small reflection principles tells us that, for every $n$, a theory $V$ proves $\opr_{V,(n)} \psi \to \psi$.
The formalised version of small reflection is internally verifiable in extensions of \eac. See \cite{verb:small94} for a study of small reflection
in the context of weak theories.

\begin{proof}
Consider any number $n$. We work in $U + \varphi$. Suppose $\sigma \leq \num n$ and $\gopr_{\phi, (\num n)}\sigma$.
Then $\sigma$ is standard (i.o.w., we can replace the bounded existential quantifier by a big disjunction),  and
 $\opr_{\phi, (F(\num n))}\sigma$. So, by small reflection, we have $\sigma$, and, hence, $\True_{\Sigma^0_1}(\sigma)$.

This simple argument can clearly be verified in \eac.
\end{proof}

\begin{theorem}\label{oppervlakkigesmurf}
\begin{enumerate}[1.]
\item
Suppose $U\vdash \phi \to \psi$. Then, $U +\phi \vdash \forall x\, ({\mf S}_\phi(x) \to {\mf S}_\psi(x))$. 
\item
Suppose $U\vdash \phi \iff \psi$. Then, $U +\phi \vdash \forall x\, ({\mf S}_\phi(x) \iff {\mf S}_\psi(x))$. 
\end{enumerate}

\smallskip\noindent
These results can be verified in \eac.
\end{theorem}

\begin{proof}
Ad (1): Suppose $U\vdash \phi \to \psi$. Let $p$ a code of a
  $U$-proof of $( \phi \to \psi)$. 
  We reason in $U+\phi$. 
  
  Suppose $x < \num p$.
In that case, we have both  ${\mf S}_\phi(x)$ and ${\mf S}_\psi(x)$, since $x$ is (externally) standard and we have both $\phi$ and $\psi$ and, thus,
we may apply Theorem~\ref{kleinesmurf}.

Suppose $x \geq \num p$. In this case, we are immediately done  by Theorem~\ref{bravesmurf}.

Case (2) is similar. 
\end{proof}

\subsection{Uniform Slow Provability}\label{sluwesmurf}
This section gives the main argument of this section. The argument is an adaptation of the argument in
\cite{viss:abso21}. Again $U$ will be our base theory as introduced in Section~\ref{prinsbsmurf}.

We define the \emph{uniform slow $(U+\phi)$-provability} of $\psi$ or $\apr_\phi \psi$ as: $\psi$ is provable from $\phi$-small $(U+\phi)$-axioms. 
We give the formal definition.
\begin{itemize}
\item
$\opr_{\phi,x}\psi$ iff there is a proof of $\psi$ from $(U+\phi)$-axioms with G\"odel numbers that are less than or equal to $x$.
Since our axiom set is elementary, the set of axioms less than or equal to $x$ is also elementary with parameter $x$.
\item
$ \apr_\phi \psi :\iff \exists x\, (\opr_{\phi,x}\psi \wedge {\mf S}_\phi(x))$.
\item 
$\aco_\phi\psi :\iff \neg\,\apr_\phi\neg\,\psi$.
\end{itemize}

We note that, in case $\phi$ is consistent with $U$ the internally defined set of small axioms numerates the $(U+\phi)$-axioms over $U+\phi$.
We remind the reader that $\alpha$ is the elementary representation of the given axiom set of $U$.
Let \[\beta(x) := ((\alpha(x)\vee x=\gnum\phi) \wedge \mf S_\phi(x)).\] Then, $\apr_\phi \chi$ is $\opr_{\tupel\beta}\chi$. 
Thus, $\apr_\phi$ is a Fefermanian predicate based on an enumeration of the axioms of
$U+\phi$. 
Clearly,
$\apr_\phi$ is $\Sigma_1^0$ modulo (\eac)-provable equivalence.  Since $U$ extends \eac, we also have this result $U$-internally.

\begin{remark}
\emph{Caveat emptor:}
 $\apr_\phi\psi$ is \emph{not} the same as $\apr_\top(\phi \to \psi)$.
 See Appendix~\ref{gastsmurf}.
 \end{remark}

 \begin{remark}
 The notion $\apr_\phi\psi$ is a kind of implication from $\phi$ to $\psi$. 
     In Appendix~\ref{robinhoodsmurf}, we briefly consider an arrow notation 
     for $\apr_\phi\psi$, to wit $\phi\tto\psi$, to get a sense of the heuristics provided by that notation.
 \end{remark}

\begin{theorem}\label{loebsmurf}
$\apr_\phi$ satisfies the L\"ob Conditions over $U+\phi$.
This result can be verified in \eac.
\end{theorem}

\begin{proof}
By Theorem~\ref{kleinesmurf}, all standard axioms of $U+\phi$ are $(U+\phi)$-provably small.
It follows that we have L\"ob's Rule, also known as Necessitation.

We note that, by Necessitation, we have \eac\ inside $\apr_\phi$ according to 
$U+\phi$.

It is immediate that $U+\phi \vdash (\apr_\phi\psi \wedge \apr_\phi (\psi\to\chi)) \to \apr_\phi\chi$.

Clearly, $U+\phi$ proves that $\apr_\phi\psi$ is equivalent to a $\Sigma^0_1$-sentence, say $\sigma$.
Moreover, by necessitation, we have $U+\phi \vdash \apr_\phi(\apr_\phi\psi \iff \sigma)$. 
It follows that:\qedright
\begin{eqnarray*}
U+ \phi \vdash \apr_\phi\psi & \to & \sigma\\
& \to & \apr_\phi \sigma \\
& \to & \apr_\phi \apr_\phi \psi.
\end{eqnarray*}
 \end{proof}
 
 \begin{remark}\label{hacksmurf}
Suppose we arrange that every proof is larger than an induction axiom that implies \eac.
Then, the axiom-set for $\apr_\phi$ will \eac-verifiably contain an axiom that implies \eac.

So, we will have L\"ob's Rule over \eac. Note that we do not need to have L\"ob's Rule
over $U$ itself.

Also, the axiom set will be $\Sigma^0_1$ modulo provable equivalence in the $\apr_\phi$-theory. This will give us full L\"ob's logic
over \eac.
\end{remark}

\begin{theorem}\label{extensionalitysmurf}
Suppose $U \vdash \phi \iff \psi$. Then, $U+\phi \vdash \forall \chi\, (\apr_{\phi}\chi \iff \apr_\psi \chi)$.
This result can be verified in \eac.
\end{theorem}

We remind the reader that we assume monotonicity of coding.

\begin{proof}
We note that $\psi$ itself is, $(U+\psi)$-provably, a $\psi$-small axiom, by Theorem~\ref{kleinesmurf}.
So, (a) $U+\psi \vdash \apr_\psi\psi$.
Let $p$ be (a code of) a $U$-proof of  $ \phi \iff \psi$. 
We note that all axioms in $p$ are standard. So, again by Theorem~\ref{kleinesmurf}, we find (b)  $U+\psi\vdash \apr_\psi(\phi\iff \psi)$.
Combining, (a) and (b), we find $U+\psi\vdash \apr_\psi\phi$.

We reason in $U+\phi$. Since, we have $\phi$ and $\phi \iff \psi$, we may conclude $\psi$. So, we have $\apr_\psi\phi$.
Consider any $\chi$ and suppose $q$ witnesses that $\apr_\phi\chi$. Thus, $q$ is  a proof of $\chi$ from $\phi$ plus $\phi$-small $U$-axioms. 
Let $r$ be a witness of $\apr_\psi\phi$.
Thus, we can transform $q$ to a witness $q'$ of $\apr_\psi\chi$, by replacing $\phi$ as an axiom by the
proof $r$. We note that $\phi$-small $U$-axioms used in $q$ are also $\psi$-small  by Theorem~\ref{oppervlakkigesmurf}.
 It follows that $\apr_\psi\chi$.

The other direction is similar.
\end{proof}

The next two theorems give the salient provability principles that our development delivers for the
combination of $\opr$ and $\apr$.

\begin{theorem}[Emission]\label{emission}
$\eac\vdash \opr_\phi \psi \to \opr_\phi \apr_\phi \psi$.  
\end{theorem}

\begin{proof}
We reason in \eac. Suppose $\opr_\phi\psi$. 
Then, clearly, for some $x$, we have $\opr_{\phi,x} \psi$.
Hence, $\opr_\phi \opr_{\phi,x}\psi$. Also, Theorem~\ref{kleinesmurf} gives us $\opr_\phi\mf S_\phi(x)$.
So, we have $\opr_\phi (\opr_{\phi,x}\psi \wedge \mf S_\phi(x))$ and, thus, $\opr_\phi\apr_\phi \psi$.
\end{proof}

 \begin{theorem}[Absorption]\label{absorption}
$\eac \vdash \opr_\phi\apr_\phi \psi \to \opr_\phi \psi$.  
 \end{theorem}
 
 The proof turns out to be remarkably simple. 
 
 \begin{proof}
We find $\nu$ such that $\eac \vdash \nu \iff  (\exists x \, \opr_{\phi,x} \psi ) < \opr_\phi \nu$.
 We note that  $\nu$ is $\Sigma^0_1$.

We reason in \eac. 
Suppose $\opr_\phi\apr_\phi \psi$. We prove $\opr_\phi \psi$.

We reason inside $\opr_\phi$. We have  $\apr_\phi \psi$.
 So, for some $x$, (i) $\opr_{\phi,x} \psi$ and (ii) $x$ is $\phi$-small.
In case not $\opr_{\phi,(x)} \nu$, by (i) and the fixed point equation, we find $\nu$. 
If we do have $\opr_{\phi,(x)} \nu$, we find $\nu$ by (ii)
and the fact that $\phi$-small proofs are $\Sigma^0_1$-reflecting
(by Theorem~\ref{mageresmurf}(1)).
 We leave the $\opr_\phi$-environment.
We have shown $\opr_\phi \nu$. 

It follows, (a) that for some $p$, we have $\opr_\phi \opr_{\phi,(p)} \nu$ and, by the fixed point equation 
for $\nu$, (b) $\opr_\phi ((\exists x\, \opr_{\phi,x} \psi) < \opr_\phi \nu)$. 
Combining (a) and (b), we find $\opr_\phi\opr_{\phi,p} \psi$, and, thus,
since $U+\phi$ is, (\eac)-verifiably, essentially reflexive, we obtain  $\opr_{\phi} \psi$, as desired. 
\end{proof}

\noindent
We provide some notes on the proof of Theorem~\ref{absorption} in Appendix~\ref{emisabso}.

\begin{theorem}\label{maintheorem}
Suppose $U$ is, $(\eac)$-verifiably, an extension of {\sf PA}. We have:
\begin{description}
\item[\strind] If $U+ \phi$ is consistent, then, $U+ \phi+ \aco_\phi \top$ is consistent and
$U+\phi + \apr_\phi\bot$ is $\Pi^0_1$-conservative over $U+\phi$.
\item[Conditional Extensionality] If $U \vdash \phi\iff \psi$, then  $U+ \phi \vdash \aco_\phi\top \iff \aco_\psi\top$.
\end{description}

\smallskip\noindent
These results can be verified in \eac.
\end{theorem}

\begin{proof}
Suppose $U+ \phi$ is consistent. Then, 
$U+\phi \nvdash \apr_\phi\bot$, by the meta-version of Theorem~\ref{absorption}.
We  apply Theorem~\ref{kreiselsmurf} to obtain the $\Pi^0_1$-conservativity of $U+\phi+\neg\, \apr_\phi \bot$ over
$U+\phi$.\footnote{We remind the reader that $\apr+\phi$ is a $\Sigma_1$ Fefermanian provability predicate for $U+\phi$, so 
Theorem~\ref{kreiselsmurf} does indeed apply.}
Conditional Extensionality follows by Theorem~\ref{extensionalitysmurf}.
\end{proof}

\subsection{Density Revisited}\label{densesmurf}
As a corollary to Theorem \ref{maintheorem}, we obtain a strengthening of Shavrukov and 
Visser's result (Theorem \ref{thm:SV_density}) on the density of the Lindenbaum algebras. 
Let $U$ be a consistent c.e.~extension of $\PA$. 
Let  $\Gamma$ be an element of  $\{\Sigma_n, \Pi_n \mid n \geq 1\}$. 

\begin{cor}
Suppose $U$ is, $(\eac)$-verifiably, an extension of {\sf PA}. We write $[\phi]$ for $[\phi]_U$.
There exists a computable function $f$ satisfying the following conditions:
\begin{description}
	\item [Strong Density] If $[\varphi] < [\psi]$ and $\varphi, \psi \in \Gamma \supseteq \Pi_1$, then we 
	have $f(\varphi, \psi) \in \Gamma$, $[\varphi] < [f(\varphi, \psi)] < [\psi]$, and for any $\Sigma_1$-sentence 
	$\sigma$, if $[\sigma] \leq [f(\varphi, \psi)]$, then $[\sigma] \leq [\varphi]$. 
	\item [Extensionality] If $[\varphi] = [\varphi']$ and $[\psi] = [\psi']$, then $[f(\varphi, \psi)] = [f(\varphi', \psi')]$.  
\end{description}
\end{cor}
\begin{proof}
Let $\rho(x)$ be a $\Pi_1$-formula having the following properties:
\begin{description}
	\item [Strong independence] If $U + \varphi$ is consistent, then $U + \varphi \nvdash \neg\,  \rho(\ulcorner \varphi \urcorner)$ and 
	$U + \varphi + \neg\,  \rho(\ulcorner \varphi \urcorner)$ is $\Pi_1$-conservative over $U + \varphi$. 
	\item [Conditional extensionality] If $U \vdash \varphi \leftrightarrow \psi$, then $U + \varphi \vdash \rho(\ulcorner \varphi \urcorner) \leftrightarrow \rho(\ulcorner \psi \urcorner)$. 
\end{description}

Let $f(\varphi, \psi)$ be the sentence
\[
	\varphi \lor (\psi \land \rho(\ulcorner \psi \land \neg\,  \varphi \urcorner)). 
\]

Suppose $[\varphi] < [\psi]$ for $\Gamma$ sentences $\varphi$ and $\psi$, where $\Gamma \supseteq \Pi_1$. 
Then $f(\varphi, \psi)$ is also a $\Gamma$ sentence. 
It is shown $[\varphi] < [f(\varphi, \psi)] < [\psi]$ in the usual way because $\rho(\ulcorner \psi \land \neg\,  \varphi \urcorner)$ is independent over $U + \psi + \neg\,  \varphi$. 

Let $\sigma$ be any $\Sigma_1$ sentence such that $[\sigma] \leq [f(\varphi, \psi)]$. 
This means
\[
	U \vdash \sigma \to \varphi \lor (\psi \land \rho(\ulcorner \psi \land \neg\,  \varphi \urcorner)). \tag{*}
\]
Then, $U + \psi \land \neg\,  \varphi + \neg\,  \rho(\psi \land \neg\,  \varphi) \vdash \neg\,  \sigma$.
By the strong independence, $U + \psi \land \neg\,  \varphi \vdash \neg\,  \sigma$. 
We have $U \vdash \sigma \land \psi \to \varphi$. 
Since $U \vdash \sigma \land \neg\,  \psi \to \varphi$ by (*), we obtain $U \vdash \sigma \to \varphi$, and so $[\sigma] \leq [\varphi]$. 

Suppose $[\varphi] = [\varphi']$ and $[\psi] = [\psi']$. 
It suffices to show $U \vdash f(\varphi, \psi) \to f(\varphi', \psi')$. 
\begin{itemize}
	\item Since $U \vdash \varphi \leftrightarrow \varphi'$, we have $U \vdash \varphi \to f(\varphi', \psi')$. 
	\item Since $U + \psi \land \neg\,  \varphi \vdash \rho(\ulcorner \psi \land \neg\,  \varphi \urcorner) \to \rho(\ulcorner \psi' \land \neg\,  \varphi' \urcorner)$, 
	we have $U \vdash \psi \land \neg\,  \varphi \land  \rho(\ulcorner \psi \land \neg\,  \varphi \urcorner) \to f(\varphi', \psi')$. 
\end{itemize}
Therefore, we conclude $U \vdash f(\varphi, \psi) \to f(\varphi', \psi')$. 
\end{proof}

Extensionality is close to the limit of what can hold.
Indeed, one of the consequences of Montalb\'an and Walsh's result shows that computable density functions cannot be monotone. 

\begin{theorem}[Montalb\'an and Walsh \cite{Mont:inev19}]
There is no computable function $h$ satisfying the following properties. 
\begin{description}
\item[Density] If $[\varphi] < [\psi]$, then $[\varphi] < [h(\varphi,\psi)] < [\psi]$.
\item[Monotonicity] If $[\varphi] \le [\varphi']$ and $[\psi] \le [\psi']$, then $[h(\varphi,\psi)]\le [h(\varphi',\psi')]$.
\end{description}
\end{theorem}

We improve this theorem by modifying Hamkins' argument presented in Appendix \ref{gastsmurf}. 

\begin{theorem}\label{Lindenbaum}
There is no function $f$ \textup(not necessarily computable\textup) satisfying the following conditions: 
For any $\varphi, \varphi', \psi, \psi' \in \mathcal{B}(\Sigma_1)$,
\begin{description}
	\item [$\mathcal{B}(\Sigma_1)$-density] If $[\varphi] < [\psi]$, then we have $f(\varphi, \psi) \in \mathcal{B}(\Sigma_1)$ and $[\varphi] < [f(\varphi, \psi)] < [\psi]$. 
	\item [$\mathcal{B}(\Sigma_1)$-monotonicity] If $[\varphi] = [\varphi']$ and $[\psi] \leq [\psi']$, then $[f(\varphi, \psi)] \leq [f(\varphi', \psi')]$.  
\end{description}
\end{theorem}
\begin{proof}
Suppose that there were such a function $f$. 
Since $[\bot] < [\top]$, we have $f(\bot, \top) \in \mathcal{B}(\Sigma_1)$ and $[\bot] < [f(\bot, \top)] < [\top]$ by the density. 
We have \[[\bot] < [\neg\,  f(\bot, \top)],\] because $[f(\bot, \top)] < [\top]$.  
Then, by the density, $[\bot] < [f(\bot, \neg\,  f(\bot, \top))] < [\neg\,  f(\bot, \top)]$.

On the other hand, since $[\bot] = [\bot]$ and $[\neg\,  f(\bot, \top)] \leq [\top]$, we obtain 
\[[f(\bot, \neg\,  f(\bot, \top))] \leq [f(\bot, \top)],\]
by the monotonicity. 
By combining this with $[f(\bot, \neg\,  f(\bot, \top))] < [\neg\,  f(\bot, \top)]$, we have $[f(\bot, \neg\,  f(\bot, \top))] \leq [\bot]$. 
This contradicts $[\bot] < [f(\bot, \neg\,  f(\bot, \top))]$. 
\end{proof}

\section{On intensionally finite Extensions}\label{inerf}
In this section, we show that the effective if-essential incompleteness of $\EA$ is inherently intensional. 
Actually, we prove more. 

\begin{theorem}\label{if1}
    Let $U$ be any c.e.~theory. 
    Suppose that a computable function $\Phi$ satisfies the following condition: 
    \begin{itemize}
        \item [$(\ast)$] For any natural number $i$, if $\mathsf{W}_i$ is a consistent finite extension of $U$, 
        then $\Phi(i)$ converges to a sentence which is independent of $\mathsf{W}_i$. 
    \end{itemize}
    Then, for any natural number $i$, if $\mathsf{W}_i$ is a consistent finite extension of $U$, then there exists a 
    natural number $j$ such that $\mathsf{W}_i = \mathsf{W}_j$, $\mathsf{W}_i \nvdash \Phi(i) \to \Phi(j)$, and $\mathsf{W}_i \nvdash \Phi(j) \to \Phi(i)$. 
\end{theorem}
\begin{proof}
Suppose $\Phi$ satisfies the condition $(\ast)$. 
Let $i$ be an index of a consistent finite extension of $U$.  
Then, $(\ast)$ implies that $\Phi(i)$ converges, $\mathsf{W}_i \nvdash \Phi(i)$, and $\mathsf{W}_i \nvdash \neg\, \Phi(i)$. 
By the usual Recursion Theorem, we find a natural number $j$ satisfying the following equation: 
{\small
\[
    \mathsf{W}_j = \begin{cases} \mathsf{W}_i + \neg \, \Phi(j) & \ \text{if}\ \Phi(j){\downarrow}\, 
    \ \text{and}\ \opr_{\mathsf{W}_i}(\gnum{\Phi(j) \to \Phi(i)}) < \opr_{\mathsf{W}_i}(\gnum{\Phi(i) \to \Phi(j)})\ \text{holds}, \\ 
    \mathsf{W}_i + \Phi(j) & \ \text{if}\ \Phi(j){\downarrow}\, \ 
    \text{and}\ \opr_{\mathsf{W}_i}(\gnum{\Phi(i) \to \Phi(j)}) \leq \opr_{\mathsf{W}_i}(\gnum{\Phi(j) \to \Phi(i)})\ \text{holds}, \\ 
    \mathsf{W}_i & \ \text{otherwise.} \end{cases}
\]
}%
Suppose, towards a contradiction, that $\Phi(j){\downarrow}$ and we have
at least one of $\mathsf{W}_i \vdash \Phi(i) \to \Phi(j)$ and $\mathsf{W}_i \vdash \Phi(j) \to \Phi(i)$. 

If $\opr_{\mathsf{W}_i}(\gnum{\Phi(j) \to \Phi(i)}) < \opr_{\mathsf{W}_i}(\gnum{\Phi(i) \to \Phi(j)})$ holds, then 
$\mathsf{W}_j$ is the finite extension $\mathsf{W}_i + \neg\, \Phi(j)$ of $U$. 
Since $\mathsf{W}_j \vdash \neg \, \Phi(j)$, the condition $(\ast)$ implies that $\mathsf{W}_j$ is inconsistent. 
We have $\mathsf{W}_i \vdash \Phi(j)$. 
Since $\mathsf{W}_i \vdash \Phi(j) \to \Phi(i)$, we have $\mathsf{W}_i \vdash \Phi(i)$. 
This is a contradiction. 

If $\opr_{\mathsf{W}_i}(\gnum{\Phi(i) \to \Phi(j)}) \leq \opr_{\mathsf{W}_i}(\gnum{\Phi(j) \to \Phi(i)})$ holds, 
then $\mathsf{W}_j$ is the finite extension $\mathsf{W}_i + \Phi(j)$ of $U$. 
Since $\mathsf{W}_j \vdash \Phi(j)$, the condition $(\ast)$ implies $\mathsf{W}_i \vdash \neg \, \Phi(j)$. 
Since $\mathsf{W}_i \vdash \Phi(i) \to \Phi(j)$, we have $\mathsf{W}_i \vdash \neg\, \Phi(i)$, a contradiction. 

We have proved that $\mathsf{W}_j = \mathsf{W}_i$, $\mathsf{W}_i \nvdash \Phi(i) \to \Phi(j)$, and $\mathsf{W}_i \nvdash \Phi(j) \to \Phi(i)$.
\end{proof}

\noindent
We note that, in Theorem~\ref{if1}, we did not demand that the language of $U$ is arithmetical. Nor did we ask that
$U$ interprets {\sf R} or anything of the sort. We also note that Theorem~\ref{if1}
holds vacuously if $U$ would not be computably enumerable.

We obtain that there is no consistent theory that is effectively if-essentially incomplete via a witnessing function satisfying Conditional Extensionality. 

\begin{cor}\label{if_cor1}
    For any consistent c.e.~theory $U$, 
    there is no computable function $\Phi$ satisfying the following condition.
    For any natural numbers $i$ and $j$, if $\mathsf{W}_i$ and $\mathsf{W}_j$ are deductively equivalent consistent finite extensions of $U$, then
    $\Phi(i)$ and $\Phi(j)$ converge to sentences and we have:
    \begin{description}
        \item[Independence] $\Phi(i)$ is independent of $\mathsf{W}_i$, 
        \item[Conditional Extensionality]
        $\mathsf{W}_i \vdash \Phi(i) \leftrightarrow \Phi(j)$. 
    \end{description}
\end{cor}

We say that a theory $U$ is \textit{effectively half-if-essentially incomplete} iff there exists a partial computable function $\Phi$ such that for any natural number $i$, if $\mathsf{W}_i$ is a consistent finite extension of $U$, then $\Phi(i)$ converges to a sentence, $U \nvdash \Phi(i)$, and $\mathsf{W}_i \nvdash \neg\, \Phi(i)$. 
In our paper \cite{kura:cert23}, we discussed that this notion is closely related to the creativity of $U$. 
The proof of the following theorem is similar to that of Theorem \ref{if1}, so we omit it.

\begin{theorem}\label{if2}
    Let $U$ be any c.e.~theory. 
    Suppose that a computable function $\Phi$ satisfies the following condition: 
    \begin{itemize}
        \item [$(\dagger)$] For any natural number $i$, if $\mathsf{W}_i$ is a consistent finite extension of $U$, then $\Phi(i)$ converges to a sentence, $U \nvdash \Phi(i)$, and $\mathsf{W}_i \nvdash \neg \, \Phi(i)$. 
    \end{itemize}
    Then, for any index $i$ of $U$, there exists an index $j$ of $U$ such that $U \nvdash \Phi(j) \to \Phi(i)$. 
\end{theorem}

The following corollary is a strengthening of Corollary \ref{if_cor1} stating that there is no consistent c.e.~theory that is effectively half-if-essentially incomplete via a witnessing function satisfying Conditional Extensionality. 

\begin{cor}\label{if_cor2}
    For any consistent c.e.~theory $U$, 
    there is no computable function $\Phi$ satisfying the following condition.
    For any natural numbers $i$ and $j$, if $\mathsf{W}_i$ and $\mathsf{W}_j$ are deductively equivalent consistent finite extensions of $U$, then
    $\Phi(i)$ and $\Phi(j)$ converge to sentences and we have:
    \begin{description}
        \item[Weak Independence] $U \nvdash \Phi(i)$ and $\mathsf{W}_i \nvdash \neg\, \Phi(i)$, 
        \item[Conditional Extensionality]
        $\mathsf{W}_i \vdash \Phi(i) \leftrightarrow \Phi(j)$. 
    \end{description}
\end{cor}

\appendix
\section{Failure of Monotonicity Revisited}\label{gastsmurf}

We revisit Hamkins' argument that independence and monotonicity cannot be combined.
We work over any consistent theory $U$. The axiom set of the theory may have any complexity.
We need not ask anything of the theory, not even that it contains numerals.
Suppose we have a mapping $\phi \mapsto \rho_\phi$ from $U$-sentences to $U$-sentences. There are no assumptions on the
complexity of the mapping or on the complexity of the values and the like. 
We write $\vdash$ for $U$-provability.
Suppose we have \emph{Independence:} if $U+\phi$ is consistent, then
$\nvdash \phi \to \rho_\phi$ and $\nvdash \phi \to \neg\, \rho_\phi$ and that we have
\emph{Monotonicity:} if $\vdash \phi \to \psi$, then $\vdash \rho_\phi \to \rho_\psi$.

We have $\vdash \neg\, \rho_\top \to \top$. Hence, $\vdash \rho_{\neg \,\rho_\top} \to \rho_\top$.
Ergo,  $\vdash  \neg\, \rho_\top \to \neg\,  \rho_{\neg \,\rho_\top} $. But, then it would follow by Independence that 
$\vdash \rho_\top$. Again by Independence, this is impossible.

Now let $U$ be a base theory as in Section~\ref{prinsbsmurf}.
When we take  $\rho_\phi := \aco_\phi \top$, since we have Independence, we find:
\[\nvdash  \neg\, \aco_\top\top \to \neg\,  \aco_{\neg \,\aco_\top\top}\top.\]
Rewriting this, we obtain: $\nvdash   \apr_\top\bot \to   \apr_{\apr_\top\bot}\bot$.
On the other hand, we do have: $\vdash   \apr_\top\bot \iff   \apr_\top \neg\, \apr_\top\bot$
(by the formalised Second Incompleteness Theorem).
So, $\apr_\top \bot \nvdash   \apr_\top \neg\, \apr_\top\bot \to   \apr_{\apr_\top\bot}\bot$. So, we have a counterexample
 to \[\phi \wedge \psi\vdash \apr_\phi (\psi \to \chi) \to \apr_{\phi\wedge \psi} \chi,\] with $\phi := \top$, $\psi := \apr_\top\bot$ and
 $\chi := \bot$.
 
 In contrast, we have:
 
\begin{theorem}
$U+ (\phi\wedge\psi)\vdash \apr_{\phi\wedge \psi} \chi \to \apr_\phi (\psi \to \chi)$.
\end{theorem}

\begin{proof}
We reason in $U$ plus $\phi\wedge \psi$. Suppose $p$ is a proof of $\chi$ from
$\phi \wedge \psi$ and $(\phi\wedge\psi)$-small axioms. Then,
we can find a proof $q$ of $(\psi\to \chi)$ from $\phi$ plus $(\phi\wedge \psi)$-small axioms.
By Theorem~\ref{oppervlakkigesmurf}, we find that  $(\phi\wedge \psi)$-small axioms are also
$\phi$-small. So, $q$ witnesses $\apr_\phi (\psi \to \chi)$. 
\end{proof}

\section{Arrow Notation}\label{robinhoodsmurf}
It is attractive to represent $\apr_\phi\psi$, as introduced in Section~\ref{sluwesmurf},  as a sort of implication $\phi \tto \psi$.
This is not entirely comfortable since we do not have
the transitivity of implication. Anyway, to see how the alternative notation looks,
we provide the principles we derived before and some new ones ---without any
claim of completeness. 

We will assume that we always have the conjunction of a finite axiomatisation of 
${\sf EA}+\mathrm B\Sigma_1$, say $\alpha$, as an axiom, whether we have $\phi$ or not, as in 
Remark~\ref{hacksmurf}, so that we have  (\dag) $\Sigma^0_1$-completeness unconditionally
and  that our provability predicate $\apr_\phi$ is equivalent to a Fefermanian $\Sigma^0_1$-predicate.

\begin{tabular}{lll}
\grullet & 
If $\phi \vdash \psi$, then $\phi \vdash \phi \tto \psi$. & {\footnotesize Theorem~\ref{loebsmurf}.} \\
\grullet & $\vdash ((\phi \tto \psi) \wedge (\phi \tto (\psi \to \chi))) \to (\phi \tto \chi)$. & {\footnotesize Theorem~\ref{loebsmurf}.}\\
\grullet &
$\phi \tto \psi \vdash \opr(\phi \tto \psi)$. & {\footnotesize $\phi \tto \psi$ is $\Sigma_1$.}\\
\grullet &
$\opr \psi \vdash \phi \tto \opr\psi$. & {\footnotesize $\opr\psi$ is $\Sigma_1$ and (\dag).}\\
\grullet &
$(\psi \tto \chi) \vdash \phi \tto (\psi \tto \chi)$. & {\footnotesize $\phi \tto \psi$ is $\Sigma_1$ and (\dag).}\\
\grullet &
$\ \opr(\phi \to \psi) \vdash \opr(\phi \to (\phi\tto \psi))$. & {\footnotesize Theorem~\ref{emission}.} \\
\grullet &
$ \opr(\phi \to (\phi \tto \psi)) \vdash \opr(\phi \to \psi)$. & {\footnotesize Theorem~\ref{absorption}.} \\
\grullet &
$(\phi\wedge \psi),\, ((\phi\wedge \psi) \tto \chi) \vdash \phi \tto (\psi \to \chi)$. & 
{\footnotesize Theorem~\ref{oppervlakkigesmurf} and the definition of $\apr$.}\\
\grullet &
$\phi \tto ((\phi \tto \psi) \to \psi) \vdash  \phi \tto \psi$. &  {\footnotesize Theorem \ref{loebsmurf}.}
\end{tabular}

\section{On the Proof of Absorption}\label{emisabso}
The fixed point $\nu$ with $\eac \vdash \nu \iff  (\exists x \, \opr_{\phi,x} \psi ) < \opr_\phi \nu$
 is an instance of the fixed point used in the FGH Theorem, so called after Harvey Friedman, Warren Goldfarb and
 Leo Harrington who each, independently, discovered the argument associated with the fixed point.
 However, in fact, John Shepherdson discovered the argument first. See \cite{shep:rep61}.
 
 The second half of the proof of Theorem~\ref{absorption} is simply a proof of a specific version of the
 FGH theorem. We cannot quite follow the usual argument, since the opposite of $\nu$, to wit
 $  \opr_\phi \nu \leq (\exists x \, \opr_{\phi,x} \psi)$ is not $\Sigma^0_1$. 

The first half of the proof of Theorem~\ref{absorption} brings us from $\opr_\phi \apr_\phi \psi$ to $\opr_\phi \nu$.
The crucial point in the paper is the step where $\phi$-smallness is used is in moving from $\opr_{\phi,(x)}\nu$ to $\nu$.

What about using the fixed point $\apr^\ast_\phi \psi$ with
\[\eac\vdash \apr^\ast_\phi \psi \iff (\exists x \, \opr_{\phi,x} \psi ) < \opr_\phi\apr^\ast_\phi \bot \;?\]
That would deliver a Fefermanian provability predicate with absorption for the case that $\psi := \bot$.
We can manipulate this by modifying the definition using $\gopr\,$-trickery in order to get closer to extensionality. However, we do not
 see how to get full extensionality.
\end{document}